\documentclass[12pt,centertags,oneside]{amsart}
\usepackage{amssymb}
\usepackage{amsmath,amstext,amsthm,a4,amssymb,amscd}
\usepackage{charter}
\usepackage{typearea}
\usepackage[mathscr]{eucal}
\usepackage{mathrsfs}
\usepackage{epsf}
\usepackage{extarrows}
\usepackage{hyperref}
\usepackage[a4paper,top=3cm,bottom=3cm,left=2cm,right=2cm]{geometry}
\numberwithin{equation}{section}
\allowdisplaybreaks[3]

\theoremstyle{plain}
\newtheorem{thm}{Theorem}[section]

\newtheorem{lemma}[thm]{Lemma}

\newtheorem{prop}[thm]{Proposition}

\theoremstyle{definition}
\newtheorem{defn}[thm]{Definition}

\newtheorem{condition}[thm]{Condition}
\newtheorem{notation}[thm]{Notation}

\theoremstyle{remark}
\newtheorem{rem}[thm]{Remark}
\newtheorem{remark}[thm]{Remark}
\newtheorem*{remark*}{Remark}

\numberwithin{equation}{section}

\newcommand{\cK}{{\mathcal{K}}}

\newcommand{\cO}{{\mathcal{O}}}

\newcommand{\wi}{\widetilde}
\DeclareMathOperator{\ke}{Ker}
\DeclareMathOperator{\pr}{pr}

\newcommand{\calig}[1]{\mathcal{#1}}

\newcommand\mE{\calig{E}}
\newcommand\mH{\calig{H}}
 \newcommand\mO{\calig{O}}
\newcommand\mQ{\calig{Q}} 
\newcommand\mS{\calig{S}} 
\newcommand\mK{\calig{K}}

\newcommand{\om}{\omega}

\newcommand{\ov}{\overline}
\newcommand{\abs}[1]{\lvert#1\rvert}

        \newcommand{\field}[1]{{\mathbb{#1}}}
        \newcommand{\NN}{\field{N}}

\newcommand{\supp}{\operatorname{supp}}

\newcommand{\End}{\mbox{\rm End}}

\allowdisplaybreaks

\begin{document}

\title[Generalized Bergman kernels]{Generalized Bergman kernels 
on symplectic manifolds of bounded geometry}

\author[Y. A. Kordyukov]{Yuri A. Kordyukov}
\address{Institute of Mathematics with Computing Centre\\
Ufa Federal Research Centre of 
         Russian Academy of Sciences\\
         112~Chernyshevsky str.\\ 450008 Ufa\\ Russia}
		 \email{yurikor@matem.anrb.ru}

\thanks{Y.\ K.\ supported by the Russian Science Foundation,
project no. 17-11-01004.
 }

\author[X. Ma]{Xiaonan Ma}
\address{Institut de Math\'ematiques de Jussieu--Paris Rive Gauche,
UFR de Math{\'e}matiques,
Universit{\'e} Paris Diderot - Paris 7, Case 7012,
75205 Paris Cedex 13, \,France}
\email{xiaonan.ma@imj-prg.fr}

\thanks{X.\ M.\ partially supported by NNSFC 11528103, 
ANR-14-CE25-0012-01 and funded through the Institutional Strategy 
of the University of Cologne within the German Excellence Initiative.}

\author[G. Marinescu]{George Marinescu}
\address{Universit{\"a}t zu K{\"o}ln,  Mathematisches Institut,
    Weyertal 86-90,   50931 K{\"o}ln, Germany
    \newline
    \mbox{\quad}\,Institute of Mathematics `Simion Stoilow',
	Romanian Academy,
Bucharest, Romania} \email{gmarines@math.uni-koeln.de}

\thanks{G.\ M.\ partially supported by DFG funded
project SFB TRR 191}



\begin{abstract}
We study the asymptotic behavior of the generalized Bergman kernel 
of the renormalized Bochner-Laplacian on high tensor powers of
a positive line bundle on a symplectic manifold of bounded geometry.
First, we establish the off-diagonal exponential estimate for 
the generalized Bergman kernel. As an application, we obtain 
the relation between the generalized Bergman kernel on 
a Galois covering of a compact symplectic manifold and 
the generalized Bergman kernel on the base. 
Then we state the full off-diagonal asymptotic expansion of 
the generalized Bergman kernel, improving the remainder 
estimate known in the compact case to an exponential decay. 
Finally, we establish 
the theory of Berezin-Toeplitz quantization on symplectic orbifolds
associated with the renormalized Bochner-Laplacian.

\end{abstract}

\date{}

 \maketitle
\tableofcontents
\section{Introduction}
In this article, we consider a smooth symplectic manifold $(X,\omega)$
of dimension $2n$. Let $(L,h^L)$ be a Hermitian line bundle on $X$
with a Hermitian connection
$\nabla^L : \mathscr{C}^\infty(X,L)\to 
\mathscr{C}^\infty(X,T^*X\otimes L)$.
We assume that $L$ satisfies the prequantization condition:
\begin{equation}\label{e1.1}
\frac{i}{2\pi}R^L=\omega,
\end{equation}
where $R^L=(\nabla^L)^2$ is the curvature of the connection $\nabla^L$.
Let $(E, h^E)$ be a Hermitian vector bundle on $X$
with Hermitian connection $\nabla^E$ and its curvature $R^E$.

Let $g^{TX}$ be a Riemannian metric on $X$ and $\nabla^{TX}$ be
the Levi-Civita connection of $(X, g^{TX})$.
Let $J_0 : TX\to TX$ be a skew-adjoint operator such that
\begin{equation}\label{e1.2}
\omega(u,v)=g^{TX}(J_0u,v), \quad u,v\in TX.
\end{equation}
Consider the operator $J : TX\to TX$ given by
\begin{equation}\label{e1.3}
J=J_0(-J^2_0)^{-1/2}.
\end{equation}
Then $J$ is an almost complex structure compatible
with $\omega$ and $g^{TX}$, that is, $g^{TX}(Ju,Jv)=g^{TX}(u,v)$,
$\omega(Ju,Jv)=\omega(u,v)$ for any $u,v\in TX$ and
$\omega(u,Ju)>0$ for any $u\in TX$, $u\neq0$.

Let $\nabla^{L^p\otimes E}: \mathscr{C}^\infty(X,L^p\otimes E)\to
\mathscr{C}^\infty(X, T^*X \otimes L^p\otimes E)$ be the connection
on $L^p\otimes E$ induced by $\nabla^{L}$ and $\nabla^E$.
Denote by $\Delta^{L^p\otimes E}$ the induced
Bochner-Laplacian acting on $\mathscr{C}^\infty(X,L^p\otimes E)$ by
\begin{equation}\label{e1.4}
\Delta^{L^p\otimes E}=\big(\nabla^{L^p\otimes E}\big)^{\!*}\,
\nabla^{L^p\otimes E},
\end{equation}
where $\big(\nabla^{L^p\otimes E}\big)^{\!*}:
\mathscr{C}^\infty(X,T^*X\otimes L^p\otimes E)\to
\mathscr{C}^\infty(X,L^p\otimes E)$ denotes the formal adjoint of 
the operator $\nabla^{L^p\otimes E}$.  
The \textbf{\emph{renormalized Bochner-Laplacian}} is a differential
operator $\Delta_p$ acting on $\mathscr{C}^\infty(X,L^p\otimes E)$ by
 \begin{equation}\label{e:Delta_p}
\Delta_p=\Delta^{L^p\otimes E}-p\tau,
  \end{equation}
where $\tau\in \mathscr{C}^\infty(X)$ is given by
 \begin{equation}\label{tau}
 \tau(x)=-\pi \operatorname{Tr}[J_0(x)J(x)],\quad x\in X.
 \end{equation}

The renormalized Bochner-Laplacian was introduced by
Guillemin and Uribe in \cite{Gu-Uribe}.
When $(X,\omega)$ is a K\"{a}hler manifold, it is twice the
corresponding Kodaira-Laplacian on functions
$\Box^{L^p}=\bar\partial^{L^p*}\bar\partial^{L^p}$.
 The asymptotic of the spectrum of the operator
 $\Delta_p$ as $p\to \infty$ was studied in
 \cite{BVa89,B-Uribe,Gu-Uribe, ma-ma02,MM07}.

In this article, we will suppose that $(X, g^{TX})$ is complete
and $R^L$, $R^E$, $J$, $g^{TX}$ have bounded geometry
(i.e., they and their derivatives of any order are uniformly
bounded on $X$ in the norm induced by $g^{TX}$, $h^L$
and $h^E$, and the injectivity radius of $(X, g^{TX})$ is positive).
We will also assume that
\begin{equation}\label{mu}
 \mu_0=\inf_{\substack{x\in X \\ u\in T_xX\setminus\{0\}}}
 \frac{iR^L_x(u,J(x)u)}{|u|_{g^{TX}}^2}>0.
 \end{equation}
Note that  $\lambda(x)=
\inf_{u\in T_xX\setminus\{0\}}iR^L_x(u,J(x)u)/|u|_{g^{TX}}^2$
is the smallest eigenvalue of $iR^L_x(\cdot,J(x)\,\cdot)$
with respect to $g^{TX}_x$, for $x\in X$. Thus \eqref{mu}
is a condition of uniform positivity of $R^L$ with respect to $g^{TX}$.

Since $(X,g^{TX})$ is complete, the Bochner-Laplacian and 
the renormalized Bochner-Laplacian $\Delta_p$
are essentially self-adjoint, see Theorem \ref{esa}. 
We still denote by $\Delta_p$ the unique self-adjoint extension of 
$\Delta_p:\mathscr{C}^\infty_c(X,L^p\otimes E)\to
\mathscr{C}^\infty_c(X,L^p\otimes E)$ 
acting on compactly supported smooth sections, 
and by $\sigma(\Delta_p)$ its spectrum in
$L^2(X,L^p\otimes E)$.
First, we state the following spectral gap property for the operator 
$\Delta_p$ which is a direct consequence of \cite[Lemma 1]{ma-ma15}. 

\begin{thm}\label{t:gap0}
Let $(X,\omega)$ be a symplectic manifold with a prequantum line bundle
$(L,\nabla^L,h^L)$. Let $g^{TX}$ be a complete
Riemannian metric on $X$ and let $J$ be the almost complex
structure defined by \eqref{e1.3}. Let
$(E,\nabla^E, h^E)$ be an auxiliary vector bundle on $X$.
We assume that $R^L$, $R^E$, $J$, $g^{TX}$ have bounded geometry
and \eqref{mu} holds.
Then there exists $C_L>0$ such that for any $p\in \mathbb N$
the spectrum of the renormalized Bochner-Laplacian \eqref{e:Delta_p}
satisfies
 \begin{equation}\label{e1.8}
 \sigma(\Delta_p)\subset \big[\!-C_L,C_L\big]\cup
 \big[2p\mu_0-C_L,+\infty\big).
 \end{equation}
 \end{thm}

When $X$ is compact and $E$ is the trivial line bundle,
this theorem (with a not precised constant $\mu_0$)
is the main result of Guillemin and Uribe \cite{Gu-Uribe}.
For a general vector bundle $E$, it was proved by Ma and Marinescu
\cite[Corollary 1.2]{ma-ma02}, cf.\ also \cite[Theorem 8.3.1]{MM07},
with the geometric constant $\mu_0$ given by \eqref{mu}.
The analogous theorem for the spin$^c$ Dirac operator
on a manifold of bounded geometry is stated in \cite[Lemma 1]{ma-ma15}.
Theorem~\ref{t:gap0} can be directly derived from this result,
following the proof of \cite[Corollary 1.2]{ma-ma02} (see
also \cite[Corollary 4.7]{ma-ma02} for the case of a covering of 
a compact manifold),
thus we will not repeat this proof here.
Note that there are cases when the renormalized Bochner-Laplacian
has a spectral gap even if the curvature $R^L$ degenerates at finite order
\cite[Remark 22]{MS18}.

For a Borel set $B\subset\mathbb{R}$,
we denote by $\mathcal{E}(B,\Delta_p)$ the spectral projection 
corresponding to the subset $B$.
Consider the spectral space $\mathcal H_p\subset L^2(X,L^p\otimes E)$
of $\Delta_p$
corresponding to $[-C_L,C_L]$,
\begin{equation}\label{e:hp}
\mH_p:=\operatorname{Range}\,\mE([-C_L,C_L],\Delta_p).
\end{equation}
If $X$ is compact, the spectrum of $\Delta_p$ is discrete
and $\mH_p$ is the subspace spanned by the eigensections of $\Delta_p$ 
corresponding to eigenvalues
in $[-C_L,C_L]$. Let 
\begin{equation}\label{php}
P_{\mathcal H_p}:=\mE([-C_L,C_L],\Delta_p)
:L^2(X,L^p\otimes E)\longrightarrow \mathcal H_p, 
\end{equation}
be the orthogonal projection.
Let $\pi_1$ and $\pi_2$ be the projections of $X\times X$ on
the first and second factor. The Schwartz kernel of the operator
$P_{\mathcal H_p}$ with respect to the Riemannian volume form
$dv_{X}$ is a smooth section
$P_{p}(\cdot,\cdot)\in \mathscr{C}^\infty(X\times X,
\pi_1^*(L^p\otimes E)\otimes \pi_2^*(L^p\otimes E)^*)$, 
see \cite[Remark 1.4.3]{MM07}. It is called the 
\textbf{\emph{generalized Bergman kernel}} of $\Delta_p$
in \cite{ma-ma08}, since it generalizes the Bergman kernel 
on complex manifolds.

\begin{thm}\label{t:mainPp}
Under the assumptions of Theorem \ref{t:gap0},
there exists $c>0$ such that for any $k\in \mathbb N$,
there exists $C_k>0$ such that for any $p\in \mathbb N$, 
$x, x^\prime \in X$, we have
\begin{equation}\label{e1.9}
\big|P_p(x, x^\prime)\big|_{\mathscr{C}^k}\leq C_k p^{n+\frac{k}{2}}
e^{-c\sqrt{p} \,d(x, x^\prime)}.
\end{equation}
\end{thm}

Here $d(x,x^\prime)$ is the geodesic distance and 
$|P_p(x, x^\prime)|_{\mathscr{C}^k}$ denotes the pointwise 
$\mathscr{C}^k$-seminorm of the section $P_p$ at a point 
$(x, x^\prime)\in X\times X$, which is the sum of the norms induced 
by $h^L, h^E$ and $g^{TX}$ of the derivatives up to order $k$ of 
$P_p$ with respect to the connection $\nabla^{L^p\otimes E}$ and 
the Levi-Civita connection $\nabla^{TX}$ evaluated at $(x, x^\prime)$.

For the Bergman kernel of the spin$^c$ Dirac operator associated to
a positive line bundle on a symplectic manifold of bounded geometry, 
the same type of
exponential estimate is proved in \cite[Theorem 1]{ma-ma15}
(see also the references therein for the previous results).
In \cite{ma-ma15}, the authors use the methods of
\cite{DLM04a,MM07,ma-ma08} based on the spectral gap property 
of the spin$^c$ Dirac operator,  
finite propagation speed arguments 
for the wave equation, the heat semigroup and rescaling of 
the spin$^c$ Dirac operator 
near the diagonal, 
which is inspired by the analytic localization technique of 
Bismut-Lebeau \cite{BL}. It is important in \cite{ma-ma15}
that the eigenvalues of the associated Laplacian are either $0$ 
or tend to $+\infty$. In the current situation, 
the renormalized Bochner-Laplacian has possibly different bounded
eigenvalues, which makes difficult to use the heat semigroup technique. 
So we replace the heat semigroup technique by a different approach, 
which was developed by the first author in \cite{bergman}: 
We follow essentially the general strategy 
of \cite{DLM04a,MM07,ma-ma08} but use weighted 
estimates with appropriate exponential weights as in \cite{Kor91} 
instead of the use of the heat semigroup and finite propagation speed 
arguments. In \cite{Kor18}, this approach is used to prove asymptotic decay of order $\mathcal O(e^{-c\sqrt{p}})$ for eigenfunctions of a self-adjoint Toeplitz operator with discrete wells associated with the renormalized Bochner-Laplacian in the classically forbidden region.

As an application of our proof of Theorem~\ref{t:mainPp},
we obtain the relation between the generalized Bergman kernel
on a Galois covering of a compact symplectic manifold and
the generalized Bergman kernel on the base as an analog of
\cite[Theorem 2]{ma-ma15} for the Bergman kernel of the
spin$^c$ Dirac operator.

\begin{thm}\label{t:covering}
Let $(X, \omega)$ be a compact symplectic manifold.
Let $(L,\nabla^L, h^L)$, $(E, \nabla^E, h^E)$, $g^{TX}$
be given as above. Consider a Galois covering $\pi : \widetilde X\to X$ and
let $\Gamma$ be the group of deck transformations.
Denote by $\widetilde \omega$, 
$(\widetilde L,\nabla^{\widetilde L}, h^{\widetilde L})$,
$(\widetilde E, \nabla^{\widetilde E}, h^{\widetilde E})$, 
$g^{T\widetilde X}$ be the lifts
of the above data to $\widetilde X$. Let $\widetilde \Delta_p$
be the renormalized Bochner-Laplacian  acting on
$\mathscr{C}^\infty(\widetilde X, \widetilde L^p\otimes \widetilde E)$
and $\widetilde P_{p}(\cdot,\cdot)$ be the generalized Bergman kernel
of $\widetilde \Delta_p$. There exists $p_1\in \mathbb N$
such that for any $p>p_1$ we have for any $x,x^\prime\in \widetilde X$,
\begin{equation}\label{e:sum}
\sum_{\gamma\in \Gamma}\widetilde P_p(\gamma x,x^\prime)=
P_p(\pi(x), \pi(x^\prime)).
\end{equation}
\end{thm}

This type of results has a long history.
In the category of complex manifolds it appeared in connection
with the theory of automorphic forms and Poincar\'e series
in the works of Selberg and Godement.
Earle \cite{Ea} gave a proof when $\widetilde X$
is a bounded symmetric domain
(under some hypothesis on the variation of Bergman kernels).
The second and third authors proved \eqref{e:sum}
for the Bergman kernels
associated to the spin$^c$ Dirac operator on a symplectic manifold,
in particular, in the K\"ahler case \cite[Theorem 2]{ma-ma15}.
Lu and Zelditch \cite{Lu-Z} independently proved \eqref{e:sum}
for the Bergman kernels on K\"ahler manifolds when $E=\mathbb{C}$.

As another application of the technique developed in this article, 
we extend the results on the full off-diagonal asymptotic expansion 
of the generalized Bergman kernels of the renormalized Bochner-Laplacians 
associated to high tensor powers of a positive line bundle 
over a compact symplectic manifold, obtained in \cite{lu-ma-ma,bergman},
to the case of manifolds of bounded geometry and slightly improve 
the remainder estimate in the asymptotic expansions, proving 
an exponential estimate $\mathcal O(e^{-c_0\sqrt{p}})$
instead of $\mathcal O(p^{-\infty})$ (see Theorem~\ref{t:main} below).

Finally, we study the theory of Berezin-Toeplitz quantization
on symplectic orbifolds by using as quantum spaces the spectral
spaces $\mH_p$, especially we show that the
set of Toeplitz operators forms an algebra.
Ma and Marinescu obtained first Berezin-Toeplitz quantization
on symplectic orbifolds by using as quantum spaces
the kernel of the spin$^c$ Dirac operator,
in particular, on compact complex orbifolds
\cite[Theorems 6.13, 6.16]{ma-ma08a}.
Let us note also that
Hsiao and Marinescu \cite{HM}
constructed a Berezin-Toeplitz quantization for eigenstates of
small eigenvalues in the case of complex manifolds.
For a comprehensive introduction to this subject
see \cite{ma:ICMtalk,MM07,MM11}.

The article is organized as follows. In Section~\ref{norm}, 
we collect some necessary background information on differential operators
and Sobolev spaces on manifolds of bounded geometry.
In Section~\ref{maintheorem}, we remind some results on 
weighted estimates on manifolds of bounded geometry and
prove Theorems~\ref{t:mainPp} and~\ref{t:covering}. 
Section~\ref{expansions} is devoted to the full off-diagonal 
asymptotic expansions. In Section~\ref{pbs4} we study 
Berezin-Toeplitz quantization on symplectic orbifolds.

\section{Preliminaries on differential operators and Sobolev spaces}
\label{norm}

In this section, we collect some necessary background information 
on differential operators and Sobolev spaces on manifolds of bounded 
geometry. We refer the reader to \cite{Kor91,ma-ma15} 
for more information. The novel point is that our constructions are adapted 
to a particular sequence of vector bundles $L^p\otimes E, p\in \mathbb N$. 
This concerns with a specific choice of the Sobolev norm as well as with
a slightly refined form of the Sobolev embedding theorem. 
We will keep the setting described in Introduction.

\subsection{Differential operators}
Let $\mathcal F$ be a vector bundle over $X$. Suppose that $\mathcal F$ 
is Euclidean or Hermitian depending on whether it is real or complex 
and equipped with a metric connection $\nabla^{\mathcal F}$. 
The Levi-Civita connection $\nabla^{TX}$ on $(X,g^{TX})$ and 
the connection $\nabla^{\mathcal F}$ define a metric connection
$\nabla^{\mathcal F} : \mathscr{C}^\infty(X, (T^*X)^{\otimes j}
\otimes \mathcal F)\to \mathscr{C}^\infty(X, (T^*X)^{\otimes (j+1)}
\otimes \mathcal F)$
on each vector bundle $(T^*X)^{\otimes j} \otimes \mathcal F$ for 
$j\in \mathbb N$, that allows us to introduce the operator 
$$\big(\nabla^{\mathcal F}\big)^{\!\ell} : 
\mathscr{C}^\infty(X, \mathcal F)
\to \mathscr{C}^\infty(X, (T^*X)^{\otimes \ell} \otimes \mathcal F)$$ 
for every $\ell\in \mathbb N$.
Any differential operator $A$ of order $q$ acting in 
$\mathscr{C}^\infty(X,\mathcal F)$ can be written as
\begin{equation}\label{e2.1}
A=\sum_{\ell=0}^q a_\ell\cdot \big(\nabla^{\mathcal F}\big)^{\!\ell},
\end{equation}
where $a_\ell \in \mathscr{C}^\infty(X, (TX)^{\otimes \ell})$ 
and the endomorphism 
$\cdot : (TX)^{\otimes \ell}\otimes ((T^*X)^{\otimes \ell}
\otimes \mathcal F) \to \mathcal F$ is given by the contraction.

If $\mathcal F$ has bounded geometry, we denote by 
$\mathscr{C}^k_b(X,\mathcal F)$ the space of sections 
$u\in \mathscr{C}^k(X,\mathcal F)$ such that
\begin{equation}\label{e2.2}
\|u\|_{\mathscr{C}^k_b}=\sup_{x\in X, \ell\leq k}
\Big|\big(\nabla^{\mathcal F}\big)^{\!\ell} u(x)\Big| <\infty,
\end{equation}
where $|\cdot|_x$ is the norm in
$(T^*_xX)^{\otimes \ell} \otimes \mathcal F_x$ 
defined by $g^{TX}$ and $h^\mathcal F$.
We also denote by 
$BD^q(X,\mathcal F)$ the space of differential operators $A$ 
of order $q$ in $\mathscr{C}^\infty_c(X,\mathcal F)$ with 
coefficients $a_\ell$ in $\mathscr{C}^\infty_b(X, (TX)^{\otimes \ell})$.

Usually, we will deal with families of differential operators of the 
form $$\{A_p\in BD^q(X,L^p\otimes E), p\in \mathbb N^*\}.$$ 
We will say that such a family $\{A_p\in BD^q(X,L^p\otimes E),
p\in \mathbb N^*\}$ is bounded in $p$, if
\begin{equation}\label{e2.3}
A_p=\sum_{\ell=0}^q a_{p,\ell}\cdot \Big(\frac{1}{\sqrt{p}}
\nabla^{L^p\otimes E}\Big)^\ell,\quad a_{p,\ell} 
\in \mathscr{C}^\infty_b(X, (TX)^{\otimes \ell}),
\end{equation}
and, for any $\ell=0,1,\ldots,q$, the family
$\{a_{p,\ell}, p\in \mathbb N^*\}$
is bounded in the Frechet space 
$\mathscr{C}^\infty_b(X, (TX)^{\otimes \ell})$.
An example of a bounded in $p$ family of differential operators
is given by $\{\frac{1}{p}\Delta_p:p\in \mathbb N^*\}$.

\subsection{Sobolev spaces}
Denote by $dv_{X}$ the Riemannian volume form of $(X,g^{TX})$. 
The $L^2$-norm on $L^2(X,L^p\otimes E)$ is given by
\begin{equation}\label{e2.5}
\|u\|^2_{p,0}=\int_{X}|u(x)|^2dv_{X}(x), \quad u\in L^2(X,L^p\otimes E).
\end{equation}
For any integer $m>0$, we introduce the norm $\|\cdot\|_{p,m}$ 
on $\mathscr{C}^\infty_c(X,L^p\otimes E)$ by the formula
\begin{equation}\label{e2.6}
\|u\|^2_{p,m}=\sum_{\ell=0}^m \int_{X} \left|\Big(\frac{1}{\sqrt{p}}
\nabla^{L^p\otimes E}\Big)^\ell u(x)\right|^2 dv_{X}(x), 
\quad u\in H^m(X,L^p\otimes E).
\end{equation}
The completion of $\mathscr{C}^\infty_c(X,L^p\otimes E)$
with respect to $\|\cdot\|_{p,m}$ is the Sobolev space 
$H^m(X,L^p\otimes E)$ of order $m$.
Denote by $\langle\cdot,\cdot\rangle_{p,m}$ the corresponding inner 
product on $H^m(X,L^p\otimes E)$. For any integer $m<0$, 
we define the norm in the Sobolev space $H^m(X,L^p\otimes E)$ by duality.
For any bounded linear operator 
$A : H^m(X,L^p\otimes E)\to H^{m^\prime}(X,L^p\otimes E)$, 
$m,m^\prime\in \mathbb Z$, we will denote its operator norm by 
$\|A\|^{m,m^\prime}_p$.

One can easily derive the following mapping properties of 
differential operators in Sobolev spaces.

\begin{prop}\label{p:Sobolev-mapping}
Any operator $A\in BD^q(X,L^p\otimes E)$ defines a bounded operator 
\[A:H^{m+q}(X,L^p\otimes E)\longrightarrow H^{m}(X,L^p\otimes E)\]
for any $m\in \mathbb N$. Moreover, if a family 
$\{A_p\in BD^q(X,L^p\otimes E), p\in \mathbb N\}$
is bounded in $p$, then for any $m\in \mathbb N$, there exists 
$C_m>0$ such that, for all $p\in \mathbb N$,
\begin{equation}\label{e2.7}
\|A_pu\|_{p,m}\leq C_m\|u\|_{p,m+q},\quad 
u\in H^{m+q}(X,L^p\otimes E).
\end{equation}
\end{prop}

\subsection{Sobolev embedding theorem}
We will need a refined form of the Sobolev embedding theorem adapted 
to the sequence $L^p\otimes E, p\in \mathbb N$.

\begin{prop}[\cite{ma-ma15}, Lemma 2]\label{p:Sobolev}
For any $k, m\in \mathbb N$ with $m>k+n$, we have an embedding
\begin{equation}\label{e2.16}
H^m(X,L^p\otimes E)\subset \mathscr{C}^k_b(X,L^p\otimes E).
\end{equation}
Moreover, there exists $C_{m,k}>0$ such that, for any $p\in \mathbb N^*$
and $u\in H^m(X,L^p\otimes E)$,
\begin{equation}\label{e2.17}
\|u\|_{\mathscr{C}^k_b}\leq C_{m,k}p^{(n+k)/2}\|u\|_{p,m}.
\end{equation}
\end{prop}

For any $x\in X$ and $v\in (L^p\otimes E)_x$, we define the delta-section 
$\delta_v\in \mathscr{C}^{-\infty}(X,L^p\otimes E)$ as a linear functional
on $\mathscr{C}^{\infty}_c(X,L^p\otimes E)$ given by
\begin{equation}\label{e2.18}
\langle \delta_v, \varphi\rangle 
=\langle v, \varphi(x)\rangle_{h^{L^p\otimes E}}, 
\quad \varphi \in \mathscr{C}^{\infty}_c(X,L^p\otimes E).
\end{equation}

\begin{prop}\label{p:delta}
For any $m>n$ and $v\in L^p\otimes E$,
$\delta_v\in H^{-m}(X,L^p\otimes E)$ with the following norm estimate
\begin{equation}\label{e2.19}
\sup_{|v|=1}p^{-n/2}\|\delta_v\|_{p,-m}<\infty.
\end{equation}
\end{prop}

\begin{proof}
By Proposition~\ref{p:Sobolev} and the definition of the Sobolev norm, 
we have
\begin{equation}\label{e2.20}
\|\delta_v\|_{p,-m}\leq C\sup_{\phi\in H^m(X,L^p\otimes E)}
\frac{\langle \delta_v,\phi\rangle}{\|\phi\|_{p,m}}\leq Cp^{n/2}|v|.
\qedhere
\end{equation}
\end{proof}

\subsection{The renormalized Bochner-Laplacian}
Let us first note the following basic result.
\begin{thm}\label{esa}
Let $(X,\omega)$ be a symplectic manifold with a prequantum line bundle
$(L,\nabla^L,h^L)$. Let $g^{TX}$ be a complete
Riemannian metric on $X$ and let
$(E,\nabla^E, h^E)$ be an auxiliary vector bundle.

\noindent
(i) 
The space $\mathscr{C}^\infty_c(X,L^p\otimes E)$ is dense in the graph
norm of the maximal extension of $\nabla^{L^p\otimes E}$ and 
$\mathscr{C}^\infty_c(X,T^*X\otimes L^p\otimes E)$
is dense in the graph norm of the maximal extension of
$\big(\nabla^{L^p\otimes E}\big)^{\!*}$.

\noindent
(ii) The Hilbert space adjoint of the maximal extension
of $\nabla^{L^p\otimes E}$ coincides with the maximal extension
of $\big(\nabla^{L^p\otimes E}\big)^{\!*}$.

\noindent
(iii) The Bochner-Laplacian 
$\Delta^{L^p\otimes E}
=\big(\nabla^{L^p\otimes E}\big)^{\!*}\,\nabla^{L^p\otimes E}$
acting on $\mathscr{C}^\infty_c(X,L^p\otimes E)$ is essentially selfadjoint.
In particular, the renormalized Bochner-Laplacian
$\Delta_p$ acting on $\mathscr{C}^\infty_c(X,L^p\otimes E)$
is essentially selfadjoint. 
\end{thm}
\begin{proof}
Assertion (i) is a form of 
the Andreotti-Vesentini Lemma \cite[Lemma 3.3.1]{MM07}.
The proof is obtained 
by replacing $\overline\partial^E$ in \cite[Lemma 3.3.1]{MM07} 
with $\nabla^{L^p\otimes E}$.
Assertions (ii) and (iii) are obtained by adapting in the same way the
proofs of \cite[Corollary 3.3.3]{MM07} and \cite[Corollary 3.3.4]{MM07},
respectively (valid for
$\big(\overline\partial^E\big)^*$ and the Kodaira-Laplacian $\Box^E$).
\end{proof}
Now we establish some additional properties of the 
family $\{\frac{1}{p}\Delta_p, p\in \mathbb N\}$ of differential operators,
which is bounded in $p$.

\begin{thm}\label{uniformD}
There exist $C_2, C_3>0$ such that for any $p\in\mathbb N^*$,
$u, u^\prime\in \mathscr{C}^\infty_c(X,L^p\otimes E)$,
\begin{equation}\label{e:deltap1}
\Big\langle \frac{1}{p}\Delta_p u, u\Big\rangle_{p,0}
\geq \|u\|^2_{p,1}-C_2\|u\|^2_{p,0}\,,
\end{equation}
\begin{equation}\label{e:deltap2}
\left|\Big\langle \frac{1}{p}\Delta_p u, u^\prime\Big\rangle_{p,0}\right|
\leq C_3\|u\|_{p,1}\|u^\prime\|_{p,1}\,.
\end{equation}
\end{thm}

\begin{proof}
These estimates follow immediately from the identity
\begin{equation}\label{e2.9}
\Big\langle \frac{1}{p}\Delta_p u, u\Big\rangle_{p,0}
=\left\|\frac{1}{\sqrt{p}}\nabla^{L^p\otimes E}u\right\|^2_{p,0}
-\langle \tau u, u\rangle_{p,0}. \qedhere
\end{equation}
\end{proof}

Let $\delta$ be the counterclockwise oriented circle in $\mathbb C$ 
centered at $0$ of radius $\mu_0$.

\begin{thm}\label{Thm1.7}
There exists $p_0\in \mathbb N$ such that for any $\lambda\in \delta$ 
and $p\geq p_0$ the operator $\lambda-\frac{1}{p}\Delta_p$ is invertible 
in $L^2(X,L^p\otimes E)$, and there exists $C>0$ such that for
all $\lambda\in \delta$ and $p\geq p_0$ we have
\begin{equation}\label{e:Thm1.7}
\left\|\Big(\lambda-\frac{1}{p}\Delta_p\Big)^{-1}\right\|^{0,0}_p\leq C,
\quad
\left\|\Big(\lambda-\frac{1}{p}\Delta_p\Big)^{-1}\right\|^{-1,1}_p\leq C.
\end{equation}
\end{thm}

\begin{proof}
We will closely follow the proof of 
\cite[Theorem 4.8]{DLM04a} or \cite[Theorem 1.7]{ma-ma08} 
(cf.\ also the proof of \cite[Theorem 11.27]{BL}).
The first estimate follows from Theorem~\ref{t:gap0} and
the spectral theorem. By \eqref{e:deltap1}, we have, for 
$\lambda_0\leq -C_2$,
\begin{equation}\label{e2.11}
\Big\langle \Big(\,\frac{1}{p}\Delta_p-\lambda_0\Big)u, 
u\Big\rangle_{p,0}\geq \|u\|^2_{p,1},
\end{equation}
therefore, the resolvent $\Big(\lambda_0-\frac{1}{p}\Delta_p\Big)^{-1}$
exists and
\begin{equation}\label{e:Thm1.71}
\left\|\Big(\lambda_0-\frac{1}{p}\Delta_p\Big)^{-1}\right\|^{-1,1}_p
\leq  1.   
\end{equation}
Now we can write, for $\lambda\in \delta$ and $\lambda_0\leq -C_2$,
\begin{equation}\label{e:Thm1.72}
\Big(\lambda-\frac{1}{p}\Delta_p\Big)^{-1}
=\Big(\lambda_0-\frac{1}{p}\Delta_p\Big)^{-1}
-(\lambda-\lambda_0)\Big(\lambda-\frac{1}{p}\Delta_p\Big)^{-1}
\Big(\lambda_0-\frac{1}{p}\Delta_p\Big)^{-1}.
\end{equation}
Thus for $\lambda\in \delta$, we get from 
the first estimate of \eqref{e:Thm1.7}, 
\eqref{e:Thm1.71} and \eqref{e:Thm1.72}, 
\begin{equation}\label{e:Thm1.73}
\left\|\Big(\lambda-\frac{1}{p}\Delta_p\Big)^{-1}\right\|^{-1,0}_p
\leq   1+C|\lambda-\lambda_0|.
\end{equation}
Changing the last two factors in \eqref{e:Thm1.72} and applying
\eqref{e:Thm1.73}, we get
\begin{equation}\label{e2.15}
\left\|\Big(\lambda-\frac{1}{p}\Delta_p\Big)^{-1}\right\|^{-1,1}_p
\leq 1+|\lambda-\lambda_0|
(1+C|\lambda-\lambda_0|).
\end{equation}
The proof of Theorem \ref{Thm1.7} is completed.
\end{proof}

\section{Proof of main results}\label{maintheorem}

This section is devoted to the proofs of Theorems~\ref{t:mainPp} and
\ref{t:covering}. First, we describe a class of exponential weight functions 
as in \cite{Kor91}. Then we prove norm estimates in weighted Sobolev 
spaces for the resolvent $\big(\lambda-\frac{1}{p}\Delta_p\big)^{-m}$.
Here we follow general constructions of \cite{DLM04a,MM07,ma-ma08}, 
which are inspired by the analytic localization technique of 
Bismut-Lebeau \cite[\S 11]{BL}.
Next, we derive pointwise exponential estimates for the 
Schwartz kernel of the operator 
$\big(\lambda-\frac{1}{p}\Delta_p\big)^{-m}$ 
and its derivatives of an arbitrary order, using a refined form of 
the Sobolev embedding theorem stated in Proposition~\ref{p:Sobolev}. 
Finally, we use the formula as in \cite[(1.55)]{ma-ma08}
\begin{equation}\label{bergman-integral}
P_{\mathcal H_p}=\frac{1}{2\pi i}
\int_\delta \lambda^{m-1}
\Big(\lambda-\frac{1}{p}\Delta_p\Big)^{-m}d\lambda, \quad m\geq 1,
\end{equation}
that allows us to complete the proofs of Theorems~\ref{t:mainPp} 
and \ref{t:covering}.

\subsection{Weight functions}
Recall that $d$ denotes the distance function on $X$.
By \cite[Proposition 4,1]{Kor91}, there exists a ``smoothed  distance''  
function  $\widetilde{d}\in \mathscr{C}^\infty(X\times X)$, satisfying 
the following conditions:
\medskip\par
(1) there is $r>0$ such that
\begin{equation}\label{e:3.11}
\big\vert \widetilde{d}(x,y) - d (x,y)\big\vert  < r,\quad x, y\in X;
\end{equation}

(2) for any $k>0$, there exists $C_k>0$ such that, for any multi-index 
$\beta$ with $|\beta|=k$,
\begin{equation}\label{e:3.12}
\big\vert \partial^\beta_{x}\widetilde{d}(x,y)\big\vert < C_{k},\quad
x, y\in X,
\end{equation}
where the derivatives are taken with respect to normal coordinates 
defined by the exponential map at $x$.

Actually, we will work with a sequence of smoothed distance functions 
$\widetilde d_p$, $p\in\mathbb N$, to remove small distances effects of
smoothing. As one can easily see from the proof of 
\cite[Proposition 4.1]{Kor91}, for any $\gamma\in (0,1]$,
there exists a function $\widetilde{d}\in \mathscr{C}^\infty(X\times X)$, 
satisfying \eqref{e:3.11} with $r=\gamma$ and \eqref{e:3.12} with 
$C_k=c_k \gamma^{1-k}$, $c_k>0$ is independent of $\gamma$. 
Let us briefly describe its construction.

Let $a^X$ be the injectivity radius of $(X,g^{TX})$.
We denote by $B^{X}(x,r)$ 
and $B^{T_{x}X}(0,r)$ the open balls in $X$ and $T_xX$ with center $x$ 
and radius $r$, respectively. For any $x_0\in X$, we identify 
$B^{T_{x_0}X}(0,a^X)$ with $B^{X}(x_0,a^X)$ via the exponential map 
$\exp^X_{x_0} : T_{x_0}X\to X$. One can show that, 
for $\varepsilon\in (0,a^X)$, the geodesic distance on 
$B^{X}(x_0,\varepsilon)$ is equivalent to the Euclidean distance on 
$B^{T_{x_0}X}(0,\varepsilon)$ uniformly on $x_0\in X$: there exists 
$C>0$ such that for any $x_0\in X$ and
$Z,W\in B^{T_{x_0}X}(0,\varepsilon)$,
\begin{equation}\label{e:equiv-dist}
C^{-1}d^{T_{x_0}X}(Z,W) \leq d \left(\exp^X_{x_0}(Z), 
\exp^X_{x_0}(W)\right)\leq Cd^{T_{x_0}X}(Z,W).
\end{equation}
By \cite[Lemma 2.3]{Kor91}, for $\varepsilon <a^X/2$,
there exists a covering $\{B^X(x_j, \varepsilon)\}_{j\in \NN}$
of $X$ and $N\in \mathbb N$ such that every intersection of $N+1$
balls  $B^X(x_j, 2\varepsilon)$ is empty. Furthermore, 
by \cite[Lemma 2.4]{Kor91}, there exists a partition of unity 
$\sum_{j=1}^\infty\phi_j = 1$ subordinated to this covering such that 
${\rm supp}\ \phi_j\subset B^X(x_j, \varepsilon)$ for any $j$ and, 
for any $k\in \mathbb N$, there exists $C_k$ such that 
$|\partial^\alpha \phi_j(x)|<C_k$ for any $j$, $x\in  B^X(x_j, \varepsilon)$
and $|\alpha|<k$, where the derivatives are computed with respect to 
the normal coordinates on $B^X(x_j, \varepsilon)$. Choose a function 
$\theta\in \mathscr{C}^\infty_c(\mathbb R^{2n})$ 
such that $\theta(x) = 0$
if $|x|>1$, $\theta(x) \geq 0$ for any $x\in \mathbb R^{2n}$ and 
$\int_{\mathbb R^{2n}} \theta(x)\,dx= 1$. For any $\delta>0$, put
\begin{equation}\label{e:3.14}
\theta_\delta(x)=\delta^{-2n}\theta(\delta^{-1}x), 
\quad x\in \mathbb R^{2n}.
\end{equation}
The function $\widetilde d$ is defined by
\begin{equation}\label{e:3.13}
\widetilde{d}(x,y)=\sum_{j=1}^\infty \phi_j(x) \int_{\mathbb R^{2n}} 
\theta_\delta\big((\exp^X_{x_j})^{-1}(x)-z\big)
d\big(\exp^X_{x_j}(z),y\big)\,dz.
\end{equation}
Using the formula
\begin{equation}\label{e:3.7}
d(x,y)=\sum_{j=1}^\infty \phi_j(x) \int_{\mathbb R^{2n}} 
\theta_\delta((\exp^X_{x_j})^{-1}(x)-z)d(x,y)\,dz,
\end{equation}
the triangle inequality, the fact that ${\rm supp}\,
\theta_\delta\subset B(0,\delta)$ and \eqref{e:equiv-dist}, we get
\begin{equation}\label{e:3.8}
\big|\widetilde{d}(x,y)-d(x,y)\big| \leq \sum_{j=1}^\infty \phi_j(x) 
\int_{\mathbb R^{2n}} \theta_\delta((\exp^X_{x_j})^{-1}(x)-z)
d(x,\exp^X_{x_j}(z))\,dz\leq C\delta.
\end{equation}

Choosing $\delta<C^{-1}\gamma$, we obtain \eqref{e:3.11} 
with $r=\gamma$. 

Differentiating \eqref{e:3.13} with respect to $x$, for any multi-index $\beta$ with $|\beta|=k$, we get
\[
\partial^\beta_{x}\widetilde{d}(x,y)= \sum_{\tau\leq \beta} C_{\beta\tau} \sum_{j=1}^\infty \partial^\tau_{x} \phi_j(x) \int_{\mathbb R^{2n}} \partial^{\beta-\tau}_{x}[\theta_\delta\big((\exp^X_{x_j})^{-1}(x)-z\big)]
d\big(\exp^X_{x_j}(z),y\big)\,dz
\]
with some constants $C_{\beta\tau}>0$. Taking into account that $\sum_{j=1}^\infty\phi_j = 1$ and $\int_{\mathbb R^{2n}} \theta_\delta(x)\,dx= 1$, it is easy to see that, for $k>0$,
\[
\sum_{\tau\leq \beta} C_{\beta\tau}\sum_{j=1}^\infty \partial^\tau_{x} \phi_j(x) \int_{\mathbb R^{2n}} \partial^{\beta-\tau}_{x}[\theta_\delta\big((\exp^X_{x_j})^{-1}(x)-z\big)]\,dz=0. 
\]
As above, using these formulas, the triangle inequality, the fact that ${\rm supp}\,
\theta_\delta\subset B(0,\delta)$ and \eqref{e:equiv-dist}, for any multi-index $\beta$ with $|\beta|=k>0$, we infer
\[
|\partial^\beta_{x}\widetilde{d}(x,y)|<\sum_{\tau\leq \beta} C_{\beta\tau}\sum_{j=1}^\infty \left|\partial^\tau_{x} \phi_j(x)\right| \int_{\mathbb R^{2n}} \left|\partial^{\beta-\tau}_{x}[\theta_\delta\big((\exp^X_{x_j})^{-1}(x)-z\big)]\right|
d\big(x,\exp^X_{x_j}(z)\big)\,dz\leq C_\beta\delta^{1-k}.
\]
For $\delta$ chosen as above, this gives \eqref{e:3.12} with $C_k=c_k \gamma^{1-k}$, $c_k>0$ is independent of $\gamma$. 

We will use such a function $\widetilde{d}$ for 
$\gamma=\frac{1}{\sqrt{p}}$, $p\in \mathbb N^{*}$, 
denoting it by $\widetilde d_p$. So it satisfies the conditions:

(1) we have
\begin{equation}\label{(1.1)}
\big\vert \widetilde{d}_p(x,y) - d (x,y)\big\vert  < \frac{1}{\sqrt{p}}\;,
\quad   \text{  for any } x, y\in X;
\end{equation}

(2) for any $k>0$, there exists $c_k>0$ such that, for any multi-index 
$\beta$ with $|\beta|= k$,
\begin{equation}
\label{dist}
\left| \Big(\frac{1}{\sqrt{p}}\Big)^{k} \partial^\beta_{x}
\widetilde{d}_p(x,y)\right| < \frac{c_{k}}{\sqrt{p}}\:,\quad 
\text{  for any } x, y\in X.
\end{equation}
For any $\alpha\in \mathbb R$, $p\in \mathbb N^{*}$ and $y\in X$, 
we introduce a weight function 
$f_{\alpha,p,y} \in \mathscr{C}^{\infty}(X)$ by
\begin{equation}\label{e:3.9}
f_{\alpha,p,y}(x) = e^{\alpha \widetilde{d}_{p,y}(x)},\quad 
\text{  for } x \in X,
\end{equation}
where $\widetilde{d}_{p,y}$ is a smooth function on $X$ given by
\begin{equation}\label{e:3.10}
\widetilde{d}_{p,y}(x) = \widetilde{d}_p(x,y), \quad   \text{  for } x\in X.
\end{equation}
We don't introduce explicitly the weighted Sobolev spaces associated 
with $f_{\alpha,p,y}$. Instead, we will work with the operator families
for $p\in \mathbb N^{*},
\alpha\in \mathbb R, y\in X$,
\begin{align}\label{e:3.15}
A_{p;\alpha,y}=f_{\alpha,p,y}A_pf^{-1}_{\alpha,p,y}: 
\mathscr{C}^\infty_c(X,L^p\otimes E)\to
\mathscr{C}^{-\infty}(X,L^p\otimes E)
\end{align}
defined by an operator family 
$\{A_p: \mathscr{C}^\infty_c(X,L^p\otimes E)\to
\mathscr{C}^{-\infty}(X,L^p\otimes E), p\in \mathbb N^{*}\}$. 
In particular, the desired exponential estimate of the Bergman kernel 
will be derived from the fact that the operator
$f_{\alpha,p,y}P_{\mathcal H_p}f^{-1}_{\alpha,p,y}$ 
is a smoothing operator in the scale of Sobolev spaces.

\subsection{Weighted estimates for the renormalized Bochner-Laplacian}
Observe that, for $v\in \mathscr{C}^\infty(X,TX)$,
\begin{equation}\label{e:3.16}
\nabla^{L^p\otimes E}_{\alpha,y;v}
:=f_{\alpha,p,y}\nabla^{L^p\otimes E}_{v}f^{-1}_{\alpha,p,y}
=\nabla^{L^p\otimes E}_{v}-\alpha v(\widetilde{d}_{p,y}).
\end{equation}
Therefore, $\nabla^{L^p\otimes E}_{\alpha,y;v}\in BD^1(X,L^p\otimes E)$. 
Moreover, for any $a>0$ and $v\in \mathscr{C}^\infty_b(X,TX)$,
the family 
\[\left\{\frac{1}{\sqrt{p}}\nabla^{L^p\otimes E}_{\alpha,y;v} : 
y\in X, |\alpha|<a\sqrt{p}\right\}\] 
is a family of operators from 
$BD^1(X,L^p\otimes E)$, uniformly bounded in $p$.

This immediately implies that, if $Q\in BD^q(X,L^p\otimes E)$,
then, for any $\alpha\in \mathbb R$ and $y\in X$, the operator 
$f_{\alpha,p,y}Qf^{-1}_{\alpha,p,y}$ is in $BD^q(X,L^p\otimes E)$. 
Moreover, for any $a>0$, the family 
$\{f_{\alpha,p,y}Qf^{-1}_{\alpha,p,y} : 
y\in X, |\alpha|<a\sqrt{p}\}$ is a family of operators from 
$BD^q(X,L^p\otimes E)$, uniformly bounded in $p$.

Now the operator  $\Delta_{p;\alpha,y}:=
f_{\alpha,p,y}\Delta_{p}f^{-1}_{\alpha,p,y}$ has the form
\begin{equation}\label{Lta}
\Delta_{p;\alpha,y}=\Delta_{p}+\alpha A_{p;y}+\alpha^2B_{p;y},
\end{equation}
where $A_{p;y}\in BD^1(X,L^p\otimes E)$ and $B_{p;y}
\in BD^0(X,L^p\otimes E)$. Moreover, for any $a>0$, the families 
$\{\frac{1}{\sqrt{p}}A_{p;y} : p\in \mathbb N^{*}, y\in X\}$ and 
$\{B_{p;y} : p\in \mathbb N^{*}, y\in X\}$ are uniformly bounded in $p$.

If $\{e_j, j=1,\ldots,2n\}$ is a local frame in $TX$ on a domain 
$U\subset X$, and functions $\Gamma^{i}_{jk}\in \mathscr{C}^\infty(U), 
i,j,k=1,\ldots,2n,$ are defined by 
$\nabla^{TX}_{e_j}e_k=\sum_{i}\Gamma^{i}_{jk}e_i$, then we have
\begin{equation}\label{Delta-p}
\Delta_{p}=-g^{jk}(Z)\left[\nabla^{L^p\otimes E}_{e_j}
\nabla^{L^p\otimes E}_{e_k}- \Gamma^{\ell}_{jk}(Z)
\nabla^{L^p\otimes E}_{e_\ell}\right]-p\tau(Z),
\end{equation}
where $(g^{jk}(Z))_{j,k}$ is the inverse of the matrix $(\langle e_{i},e_{j}\rangle(Z))_{i,j}$ 
and
\begin{equation}\label{LtaZ}
\Delta_{p;\alpha,y}=-g^{jk}(Z)\left[\nabla^{L^p\otimes E}_{\alpha,y;e_j}
\nabla^{L^p\otimes E}_{\alpha,y;e_k}
- \Gamma^{\ell}_{jk}(Z)\nabla^{L^p\otimes E}_{\alpha,y;e_\ell}\right]
-p\tau(Z).
\end{equation}
In particular,
 \begin{align}\label{ReLta}
\operatorname{Re}\, \Delta_{p;\alpha,y}=\Delta_{p} 
- \alpha^2\sum_{j,k=1}^{2n} g^{jk}(Z) e_j(\widetilde{d}_{p,y})
e_k(\widetilde{d}_{p,y})
=\Delta_{p} - \alpha^2|\nabla\widetilde{d}_{p,y}|_{g(Z)}^2.
\end{align}
From \eqref{LtaZ}, we easily get
\begin{align}\begin{split}\label{Apy}
&A_{p;y}=-\sum_{j,k=1}^{2n}g^{jk}(Z)\Big(2e_j(\widetilde{d}_{p,y})
\nabla^{L^p\otimes E}_{e_k} +e_j(e_k(\widetilde{d}_{p,y}))
-\Gamma^{\ell}_{jk}(Z)e_\ell(\widetilde{d}_{p,y})\Big),\\ 
&B_{p;y}
=-\sum_{j,k=1}^{2n}g^{jk}(Z)e_j(\widetilde{d}_{p,y})
e_k(\widetilde{d}_{p,y}).
\end{split} \end{align}
By Proposition~\ref{p:Sobolev-mapping}, for any $m\in \mathbb N$, 
there exists $C_m>0$ such that, for any $p\in \mathbb N^{*}$, $y\in X$ 
and $u\in H^{m}(X,L^p\otimes E)$,
\begin{equation}\label{e:Sobolev-mapping}
\|A_{p;y}u\|_{p,m-1}\leq C_mp^{1/2}\|u\|_{p,m},\quad \|B_{p;y}u\|_{p,m}
\leq C_m\|u\|_{p,m}.
\end{equation}

We have the following extension of Theorem~\ref{uniformD}.
\begin{thm}\label{t3.1}
There exist $C_0, C_2, C_3>0$ such that for any 
$p\in \mathbb N^{*}$, $\alpha \in \mathbb R$, $y\in X$ and 
$u, u^\prime\in \mathscr{C}^\infty_c(X,L^p\otimes E)$,
\begin{align}\begin{split}\label{e:3.17}
&\operatorname{Re}\,  \Big\langle \frac{1}{p}\Delta_{p;\alpha,y} u, 
u\Big\rangle_{p,0}\geq \|u\|^2_{p,1}
-\Big(C_2+C_0\frac{\alpha^2}{p}\Big)\|u\|^2_{p,0},\\[3pt]
&\left|\Big\langle \frac{1}{p}\Delta_{p;\alpha,y} u, 
u^\prime\Big\rangle_{p,0}\right|
\leq C_3\left(\|u\|_{p,1}\|u^\prime\|_{p,1}
+ \left(\frac{|\alpha|}{\sqrt{p}}\|u\|_{p,1}
+\frac{\alpha^2}{p}\|u\|_{p,0}\right)\|u^\prime\|_{p,0}\right).
\end{split} \end{align}
\end{thm}

\begin{proof} Using Theorem~\ref{uniformD}, \eqref{Lta}, \eqref{ReLta}, 
	and \eqref{e:Sobolev-mapping}, we get
\begin{equation}\label{e:3.18}
\begin{split}
\operatorname{Re}\, \Big\langle \frac{1}{p} \Delta_{p;\alpha,y} u, 
u\Big\rangle_{p,0}
&\geq \Big\langle \frac{1}{p} \Delta_{p} u, u\Big\rangle_{p,0}
-C_0\frac{\alpha^2}{p} \|u\|^2_{p,0}\\[3pt]
&\geq \|u\|^2_{p,1}
-\Big(C_2+C_0\frac{\alpha^2}{p}\Big)\|u\|^2_{p,0}\,,
\end{split}
\end{equation}
and
\begin{equation}\label{e:3.19}
\begin{split}
&\left|\Big\langle \frac{1}{p} \Delta_{p;\alpha,y} u, 
u^\prime\Big\rangle_{p,0}\right|\\[3pt]
&\leq \left|\Big\langle \frac{1}{p} \Delta_{p} u, 
u^\prime\Big\rangle_{p,0}\right| 
+ |\alpha| \left|\Big\langle \frac{1}{p} A_{p,y} u, 
u^\prime\Big\rangle_{p,0}\right|
+ \alpha^2 \left|\Big\langle \frac{1}{p} B_{p,y} u,
u^\prime\Big\rangle_{p,0}\right|\\[3pt]
&\leq  C_3\left(\|u\|_{p,1}\|u^\prime\|_{p,1}
+\frac{|\alpha|}{\sqrt{p}}\|u\|_{p,1}\|u^\prime\|_{p,0}
+\frac{\alpha^2}{p} \|u\|_{p,0}\|u^\prime\|_{p,0}\right). 
\end{split}
\end{equation}
\end{proof}

\subsection{Weighted estimates for the resolvent}\label{s3.3}
Theorems \ref{Thm1.7W}- \ref{Thm1.9} are the weighted analogs of
\cite[Theorems 4.8-4.10]{DLM04a}, \cite[Theorems 1.7-1.9]{ma-ma08}
which are inspired by \cite[\S 11]{BL}.
Now we extend Theorem~\ref{Thm1.7} to the setting of weighted spaces.

\begin{thm}\label{Thm1.7W}
There exist $c>0$, $C>0$ and $p_0\in \mathbb N$ such that, for all 
$\lambda\in \delta$, $p\geq p_0$, $|\alpha|<c\sqrt{p}$, $y\in X$, 
the operator $\lambda-\frac{1}{p}\Delta_{p;\alpha,y}$ is invertible in
$L^2(X,L^p\otimes E)$, and
we have
\begin{align}\label{e:3.20}
\left\|\Big(\lambda-\frac{1}{p}\Delta_{p;\alpha,y}\Big)^{-1}
\right\|^{0,0}_p\leq C, \quad
\left\|\Big(\lambda-\frac{1}{p}\Delta_{p;\alpha,y}\Big)^{-1}
\right\|^{-1,1}_p\leq C.
\end{align}
\end{thm}

\begin{proof}
Let us denote in this proof
$$R(\lambda,\tfrac1p\Delta_p):=\Big(\lambda-\frac{1}{p}\Delta_{p}\Big)^{-1},
\quad R\big(\lambda,\tfrac1p\Delta_{p;\alpha,y}\big):=\Big(\lambda-\frac{1}{p}\Delta_{p;\alpha,y}
\Big)^{-1}.$$
By Theorem~\ref{Thm1.7}, \eqref{Lta} and \eqref{e:Sobolev-mapping}, 
it follows that, for all $\lambda\in \delta$, $p\in \mathbb N^{*}$, 
$\alpha\in \mathbb R$ and $y\in X$, we have
\begin{equation}\label{e:3.22}
\begin{split}
&\left\|(\Delta_{p;\alpha,y}-\Delta_{p})
R(\lambda,\tfrac1p\Delta_p)\right\|^{-1,0}_p
= \left\|\Big(\alpha A_{p,y}+\alpha^2B_{p,y}\Big)
R(\lambda,\tfrac1p\Delta_p)\right\|^{-1,0}_p\\[3pt]
&\leq  C\left(|\alpha|\sqrt{p}
\left\|R(\lambda,\tfrac1p\Delta_p)\right\|^{-1,1}_p
+\alpha^2\left\|R(\lambda,\tfrac1p\Delta_p)
\right\|^{-1,0}_p\right) 
\leq C\Big(|\alpha|\sqrt{p}+\alpha^2\Big).
\end{split}
\end{equation}
Choose $c>0$ such that $C(c+c^2)<\frac 12$. Then, if 
$|\alpha|<c\sqrt{p}$, we have
\begin{equation}\label{e:d}
\left\|\Big(\frac{1}{p}\Delta_{p;\alpha,y}-\frac{1}{p}\Delta_{p}\Big)
R(\lambda,\tfrac1p\Delta_p)\right\|^{0,0}_p
\leq \left\|\Big(\frac{1}{p}\Delta_{p;\alpha,y}-\frac{1}{p}\Delta_{p}\Big)
R(\lambda,\tfrac1p\Delta_p)\right\|^{-1,0}_p<\frac 12\,\cdot
\end{equation}
Therefore, for all $\lambda\in \delta$, $p\in \mathbb N^{*}$, 
$\alpha\in \mathbb R$, $|\alpha|<c\sqrt{p}$, and $y\in X$, 
the operator $\lambda-\frac{1}{p}\Delta_{p;\alpha,y}$ is invertible in 
$L^2$, and we have
\begin{equation}\label{e:res}
R\big(\lambda,\tfrac1p\Delta_{p;\alpha,y}\big)
=R(\lambda,\tfrac1p\Delta_p)\\ 
+R(\lambda,\tfrac1p\Delta_p)
\sum_{j=1}^{\infty}\Big(\Big(\frac{1}{p}\Delta_{p;\alpha,y}
-\frac{1}{p}\Delta_{p}\Big)R\big(\lambda,\tfrac1p\Delta_{p;\alpha,y}\big)\Big)^j.
\end{equation}
Therefore, by \eqref{e:Thm1.7}, \eqref{e:d} 
and \eqref{e:res}, we get
\begin{equation}\label{e:3.26}
\begin{split}
&\left\|R\big(\lambda,\tfrac1p\Delta_{p;\alpha,y}\big)\right\|^{-1,1}_p\leq 
\left\|R(\lambda,\tfrac1p\Delta_p)\right\|^{-1,1}_p\\ 
&+\left\| R(\lambda,\tfrac1p\Delta_p)
\right\|^{0,1}_p\sum_{j=1}^{\infty}\left\| \Big(\Big(\frac{1}{p}
\Delta_{p;\alpha,y}-\frac{1}{p}\Delta_{p}\Big)
R\big(\lambda,\tfrac1p\Delta_{p;\alpha,y}\big)\Big)^{j-1}
\right\|^{0,0}_p \\ 
& \times 
\left\| \Big(\frac{1}{p}\Delta_{p;\alpha,y}
-\frac{1}{p}\Delta_{p}
\Big)R\big(\lambda,\tfrac1p\Delta_{p;\alpha,y}\big)
\right\|^{-1,0}_p\\
&
\leq  C + C \sum_{j=1}^{\infty}2^{-j} = 2C.
\end{split}
\end{equation}
Since $\|\cdot \|^{0,0}_{p}\leq \|\cdot \|^{-1,1}_{p}$,
\eqref{e:3.26} entails \eqref{e:3.20}.
\end{proof}

In the sequel, we will keep notation $c$ for the constant given by 
Theorem \ref{Thm1.7W}, which will be usually related with the interval 
$(-c\sqrt{p}, c\sqrt{p})$ of admissible values of the parameter $\alpha$.

\begin{remark}\label{r:conj}
Observe that, for any $\lambda\in \delta$, $p\geq p_0$, 
$\alpha\in \mathbb R$ and $y\in X$, the operators 
$(\lambda-\frac{1}{p}\Delta_{p;\alpha,y})^{-1}$ and 
$(\lambda-\frac{1}{p}\Delta_{p})^{-1}$ are related by the identity
\begin{equation}\label{LaL}
\Big(\lambda-\frac{1}{p}\Delta_{p;\alpha,y}\Big)^{-1}
=f_{\alpha,p,y}\Big(\lambda-\frac{1}{p}\Delta_{p}\Big)^{-1}
f^{-1}_{\alpha,p,y},
\end{equation}
which should be understood in the following way. If $\alpha<0$, then,
for any $s\in \mathscr{C}^\infty_c(X, L^p\otimes E)$, the expression 
$f_{\alpha,p,y}(\lambda-\frac{1}{p}\Delta_{p})^{-1}f^{-1}_{\alpha,p,y}s$ 
makes sense and defines a function in $L^2(X, L^p\otimes E)$.
Thus, we get a well-defined operator 
\begin{align}\label{e:3.30}
	f_{\alpha,p,y}\Big(\lambda-\frac{1}{p}
	\Delta_{p}\Big)^{-1}f^{-1}_{\alpha,p,y}:
\mathscr{C}^\infty_c(X, L^p\otimes E)\to L^2(X, L^p\otimes E),
\end{align}
and one can check that $f_{\alpha,p,y}
(\lambda-\frac{1}{p}\Delta_{p})^{-1}f^{-1}_{\alpha,p,y}u
=(\lambda-\frac{1}{p}\Delta_{p;\alpha,y})^{-1}u$ for any 
$u\in \mathscr{C}^\infty_c(X, L^p\otimes E)$. So \eqref{LaL} 
means that the operator 
$f_{\alpha,p,y}(\lambda-\frac{1}{p}\Delta_{p})^{-1}f^{-1}_{\alpha,p,y}$
extends to a bounded operator in $L^2(X, L^p\otimes E)$, 
which coincides with $(\lambda-\frac{1}{p}\Delta_{p;\alpha,y})^{-1}$.
If $\alpha>0$, then, for any $u\in L^2(X, L^p\otimes E)$, the expression 
$f_{\alpha,p,y}(\lambda-\frac{1}{p}\Delta_{p})^{-1}f^{-1}_{\alpha,p,y}u$
makes sense as a distribution on $X$. Thus, we get a well-defined operator 
\begin{align}\label{e:3.31}
	f_{\alpha,p,y}\Big(\lambda-\frac{1}{p}
	\Delta_{p}\Big)^{-1}f^{-1}_{\alpha,p,y}: 
L^2(X, L^p\otimes E) \to \mathscr{C}^{-\infty}(X, L^p\otimes E).
\end{align}
So \eqref{LaL} means that this operator is indeed a bounded operator in 
$L^2(X, L^p\otimes E)$, which coincides with 
$(\lambda-\frac{1}{p}\Delta_{p;\alpha,y})^{-1}$.
\end{remark}

\begin{thm}\label{Thm1.9}
For any $p\in \mathbb N^{*}$, $p\geq p_0$, $\lambda\in \delta$, 
$m\in \mathbb N$, $y\in X$ and $|\alpha|<c\sqrt{p}$ with the constant 
$c$ as in Theorem \ref{Thm1.7W}, the resolvent 
$(\lambda-\frac{1}{p}\Delta_{p;\alpha,y})^{-1}$ maps 
$H^m(X,L^p\otimes E)$ to $H^{m+1}(X,L^p\otimes E)$. 
Moreover, for any $m\in \mathbb N$, there exists $C_m>0$ 
such that for any $p\geq p_0$, $\lambda\in \delta$, $y\in X$ and 
$|\alpha|<c\sqrt{p}$,
\begin{equation}\label{e:mm+1}
\left\|\Big(\lambda-\frac{1}{p}
\Delta_{p;\alpha,y}\Big)^{-1}\right\|_{p}^{m,m+1} \leq C_{m}.
\end{equation}
\end{thm}

 \begin{proof}
The first statement of the theorem is a consequence of a general fact 
about operators on manifolds of bounded geometry. The operator 
$(\lambda-\frac{1}{p}\Delta_{p;\alpha,y})^{-1}$ is
a pseudodifferential operator of order $-2$, so it maps 
$H^m(X,L^p\otimes E)$ to $H^{m+2}(X,L^p\otimes E)$.
It remains to prove the norm estimate \eqref{e:mm+1}.

To prove \eqref{e:mm+1}, first, we introduce normal coordinates 
near an arbitrary point $x_0\in X$. As above, we will identify the 
balls $B^{T_{x_0}X}(0,a^X)$ and $B^{X}(x_0,a^X)$ via the exponential 
map $\exp^X_{x_0} : T_{x_0}X\to X$. Furthermore, we choose
a trivialization of the bundle $L$ and $E$ over $B^{X}(x_0,a^X)$,  
identifying the fibers $L_Z$ and $E_Z$ of $L$ and $E$ at 
$Z\in B^{T_{x_0}X}(0,a^X)\cong B^{X}(x_0,a^X)$ with $L_{x_0}$ 
and $E_{x_0}$ by parallel transport with respect to the connection 
$\nabla^L$ and $\nabla^E$ along the curve 
$\gamma_Z : [0,1]\ni u \to \exp^X_{x_0}(uZ)$. 
Denote by $\nabla^{L^p\otimes E}$ and $h^{L^p\otimes E}$ 
the connection and the Hermitian metric on the trivial line bundle with 
fiber $(L^p\otimes E)_{x_0}$ induced by this trivialization.
Let $\Gamma^L$, $\Gamma^E$ be the connection forms of 
$\nabla^L$ and $\nabla^E$ with respect to some fixed frames for
$L$, $E$ which is parallel along the curve 
$\gamma_Z : [0,1]\ni u \to \exp^X_{x_0}(uZ)$ under our trivialization of
$B^{T_{x_0}X}(0,\varepsilon)$.
Then we have
\begin{align}\label{e:3.32}
\nabla^{L^p}_U=\nabla_U+p\Gamma^L(U)+\Gamma^E(U).
\end{align}

For any $x_0\in X$, fix an orthonormal basis $e_1,\ldots,e_{2n}$ in 
$T_{x_0}X$.  We still denote by 
$\{e_{j}\}_{j=1}^{2n}$ the constant vector fields $e_j(Z)=e_j$
on $B^{T_{x_0}X}(0,\varepsilon)$.
One can show that the restriction
of the norm $\|\cdot\|_{p,m}$ to
$\mathscr{C}^\infty_c(B^{T_{x_0}X}(0,\varepsilon), L^p\otimes E)
\cong \mathscr{C}^\infty_c(B^{X}(x_0,\varepsilon), L^p\otimes E)$ 
is equivalent uniformly on $x_0\in X$ and $p\in \mathbb N^{*}$
to the norm $\|\cdot \|^\prime_{p,m}$ given for 
$u\in \mathscr{C}^\infty_c(B^{T_{x_0}X}(0,\varepsilon), L^p\otimes E)$
by
\begin{equation}\label{localSobolev}
\|u\|^\prime_{p,m}=\Big(\sum_{\ell=0}^m\sum_{j_1,\ldots,j_\ell=1}^{2n}
\int_{T_{x_0}X} \Big(\frac{1}{\sqrt{p}}\Big)^\ell|
\nabla^{L^p\otimes E}_{e_{j_1}}\cdots 
\nabla^{L^p\otimes E}_{e_{j_\ell}}u|^2dZ\Big)^{1/2}.
\end{equation}
That is, there exists $C_m>0$ such that, for any $x_0\in X$, 
$p\in \mathbb N^{*}$ 
we have
\begin{equation}\label{pm-prime}
C_m^{-1}\|u\|^\prime_{p,m}\leq \|u\|_{p,m}\leq C_m\|u\|^\prime_{p,m}\,,
\end{equation}
for any $u\in \mathscr{C}^\infty_c(B^{T_{x_0}X}(0,\varepsilon),
L^p\otimes E)
\cong \mathscr{C}^\infty_c(B^{X}(x_0,\varepsilon), L^p\otimes E)$.
By choosing an appropriate covering of $X$ by normal coordinate charts, 
we can reduce our considerations to the local setting. Without loss
of generality, we can assume that 
$u\in \mathscr{C}^\infty_c(B^{T_{x_0}X}(0,\varepsilon), L^p\otimes E)$
for some $x_0\in X$ and the Sobolev norm of $u$ is given by
the norm $\|u\|^\prime_{p,m}$ given by \eqref{localSobolev}.
(Later on, we omit `prime' for simplicity.)
We have to show the estimate \eqref{e:mm+1}, uniform on $x_0$. 
That is, we claim that, for any $m\in \mathbb N$, there exists $C_m>0$ 
such that for any $p\in \mathbb N^{*}$, $p\geq p_0$, 
$\lambda\in \delta$, $y\in X$, $|\alpha|<c\sqrt{p}$ and $x_0\in X$,
\begin{equation}\label{e:local}
\left\|\Big(\lambda-\frac{1}{p}\Delta_{p;\alpha,y}\Big)^{-1}u
\right\|_{p,m+1}\leq C_{m}\left\|u\right\|_{p,m},\quad
u\in \mathscr{C}^\infty_c(B^{T_{x_0}X}(0,\varepsilon), L^p\otimes E).
\end{equation}

\begin{prop}\label{p:iterated}
For any $1\leq j_1 \leq \ldots \leq j_k\leq 2n$, the iterated commutator
\begin{equation}\label{e:comm}
\left[\frac{1}{\sqrt{p}}\nabla^{L^p\otimes E}_{e_{j_1}},
\left[\frac{1}{\sqrt{p}}\nabla^{L^p\otimes E}_{e_{j_2}},\ldots, 
\left[\frac{1}{\sqrt{p}}\nabla^{L^p\otimes E}_{e_{j_k}}, 
\frac{1}{p}\Delta_{p;\alpha,y}\right]\ldots \right]\right]
\end{equation}
defines a family of second order differential operators, bounded uniformly 
on $p\in \mathbb N^{*}$, $y\in X$, $|\alpha|<c\sqrt{p}$ and $x_0\in X$.
In particular, there exists $C>0$ such that for $p\in \mathbb N^{*}$, 
$|\alpha|<c\sqrt{p}$, $y\in X$ and 
$u,u^\prime\in \mathscr{C}^\infty_c(B^{T_{x_0}X}(0,\varepsilon), 
L^p\otimes E)$
\begin{multline}\label{e:3.35}
\left|\Big\langle \left[\frac{1}{\sqrt{p}}\nabla^{L^p\otimes E}_{e_{j_1}}, 
\left[\frac{1}{\sqrt{p}}\nabla^{L^p\otimes E}_{e_{j_2}},\ldots, 
\left[\frac{1}{\sqrt{p}}\nabla^{L^p\otimes E}_{e_{j_k}}, 
\frac{1}{p}\Delta_{p;\alpha,y}\right]\ldots \right]\right]u, 
u^\prime \Big\rangle_{p,0}\right| \\
\leq C\|u\|_{p,1}\|u^\prime\|_{p,1}.
\end{multline}
\end{prop}

\begin{proof}
By  \eqref{Lta}, \eqref{LtaZ} and \eqref{Apy},
the operator $\frac{1}{p}\Delta_{p;\alpha,y}$ has the form
\begin{equation}\label{e:3.36}
\begin{split}
\frac{1}{p}\Delta_{p;\alpha,y}=\sum_{i,j}a^{ij}_{p;\alpha,y}(Z)
\Big(\frac{1}{\sqrt{p}}\nabla^{L^p\otimes E}_{e_i}\Big)
\Big(\frac{1}{\sqrt{p}}\nabla^{L^p\otimes E}_{e_j}\Big)\\
+\sum_{\ell}b^{\ell}_{p;\alpha,y}(Z)\frac{1}{\sqrt{p}}
\nabla^{L^p\otimes E}_{e_\ell}+c_{p;\alpha,y}(Z),
\end{split}
\end{equation}
where
\begin{align*}
a_{p;\alpha,y}^{ij}(Z)=&-g^{ij}(Z),\\
b_{p;\alpha,y}^\ell(Z)=&\frac{1}{\sqrt{p}}\sum_{j,k=1}^{2n} g^{jk}(Z)
\Gamma^\ell_{jk}(Z)-\frac{2\alpha}{\sqrt{p}} \sum_{j=1}^{2n}
g^{j\ell}(Z)e_j(\widetilde{d}_{p,y}),\\
c_{p;\alpha,y}(Z)=&-\tau(Z)-\frac{\alpha}{p}\sum_{j,k=1}^{2n}
g^{jk}(Z)\Big(e_j(e_k(\widetilde{d}_{p,y}))- \Gamma^{\ell}_{jk}(Z)
e_\ell(\widetilde{d}_{p,y})\Big)\\ 
&-\frac{\alpha^2}{p}\sum_{j,k=1}^{2n}g^{jk}(Z)e_j(\widetilde{d}_{p,y})
e_k(\widetilde{d}_{p,y}).
\end{align*}
It is easy to see that if $f_{p;\alpha,y}$ is $a_{p;\alpha,y}^{ij}$,
$b_{p;\alpha,y}^\ell$ or $c_{p;\alpha,y}$, the iterated commutator
\[
\left[\frac{1}{\sqrt{p}}\nabla^{L^p\otimes E}_{e_{j_1}},
\left[\frac{1}{\sqrt{p}}\nabla^{L^p\otimes E}_{e_{j_2}},\ldots,
\left[\frac{1}{\sqrt{p}}\nabla^{L^p\otimes E}_{e_{j_k}}, f_{p;\alpha,y}
\right]\ldots \right]\right]
\]
is a smooth function on $B^{T_{x_0}X}(0,\varepsilon)$ with sup-norm, 
uniformly bounded on $p\in \mathbb N^{*}$, $y\in X$, 
$|\alpha|<c\sqrt{p}$ and $x_0\in X$.

Recall the commutator relations
\begin{align}\label{e:3.37}
\left[\frac{1}{\sqrt{p}}\nabla^{L^p\otimes E}_{e_i}, 
\frac{1}{\sqrt{p}}\nabla^{L^p\otimes E}_{e_j}\right]
=\frac{1}{p}R^{L^p\otimes E}(e_i,e_j)=R^{L}(e_i,e_j)
+\frac{1}{p}R^{E}(e_i,e_j).
\end{align}

Using these facts, one can see that the 
iterated commutator \eqref{e:comm}
has the same structure as $\frac{1}{p}\Delta_{p;\alpha,y}$. 
This easily completes the proof.
\end{proof}

Now the proof of \eqref{e:local} is completed as in 
\cite[Theorem 4.10]{DLM04a} or \cite[Theorem 1.9]{ma-ma08}. For any 
$1\leq j_1 \leq \ldots \leq j_\ell\leq 2n$ with $\ell=1,\ldots, m$, 
we can write the operator
\begin{align}\label{e:3.38}
\Big(\frac{1}{\sqrt{p}}\nabla^{L^p\otimes E}_{e_{j_1}}\Big) 
\Big(\frac{1}{\sqrt{p}}\nabla^{L^p\otimes E}_{e_{j_2}}\Big)\ldots 
\Big(\frac{1}{\sqrt{p}}\nabla^{L^p\otimes E}_{e_{j_\ell}}\Big) 
(\lambda-\frac{1}{p}\Delta_{p;\alpha,y})^{-1}
\end{align}
as a linear combination of operators of the type
\begin{multline}\label{e:3.39}
\left[\frac{1}{\sqrt{p}}\nabla^{L^p\otimes E}_{e_{j_1}},
\left[\frac{1}{\sqrt{p}}\nabla^{L^p\otimes E}_{e_{j_2}},\ldots, 
\left[\frac{1}{\sqrt{p}}\nabla^{L^p\otimes E}_{e_{j_k}},
(\lambda-\frac{1}{p}\Delta_{p;\alpha,y})^{-1}\right]\ldots \right]\right]
\times \\ \times \Big(\frac{1}{\sqrt{p}}
\nabla^{L^p\otimes E}_{e_{j_{k+1}}}\Big)\ldots  
\Big(\frac{1}{\sqrt{p}}\nabla^{L^p\otimes E}_{e_{j_\ell}}\Big)
\end{multline}
and of the operator
\begin{align}\label{e:3.40}
\Big(\lambda-\frac{1}{p}\Delta_{p;\alpha,y}\Big)^{-1}
\Big(\frac{1}{\sqrt{p}}
\nabla^{L^p\otimes E}_{e_{j_1}}\Big) \Big(\frac{1}{\sqrt{p}}
\nabla^{L^p\otimes E}_{e_{j_2}}\Big)\ldots \Big(\frac{1}{\sqrt{p}}
\nabla^{L^p\otimes E}_{e_{j_\ell}}\Big).
\end{align}
Each commutator 
\[
\left[\frac{1}{\sqrt{p}}\nabla^{L^p\otimes E}_{e_{j_1}},
\left[\frac{1}{\sqrt{p}}\nabla^{L^p\otimes E}_{e_{j_2}},\ldots, 
\left[\frac{1}{\sqrt{p}}\nabla^{L^p\otimes E}_{e_{j_k}},
(\lambda-\frac{1}{p}\Delta_{p;\alpha,y})^{-1}\right]\ldots \right]\right]
\]
is a linear combination of operators of the form
\begin{equation}\label{e:RR}
\Big(\lambda-\frac{1}{p}\Delta_{p;\alpha,y}\Big)^{-1}R_1
\Big(\lambda-\frac{1}{p}\Delta_{p;\alpha,y}\Big)^{-1}R_2 \ldots
R_k\Big(\lambda-\frac{1}{p}\Delta_{p;\alpha,y}\Big)^{-1},
\end{equation}
where the operators $R_1,\ldots, R_k$ are of the form
\begin{align}\label{e:3.41}
\left[\frac{1}{\sqrt{p}}\nabla^{L^p\otimes E}_{e_{i_1}}, 
\left[\frac{1}{\sqrt{p}}\nabla^{L^p\otimes E}_{e_{i_2}},\ldots,
\left[\frac{1}{\sqrt{p}}\nabla^{L^p\otimes E}_{e_{i_l}},
\frac{1}{p}\Delta_{p;\alpha,y}\right]\ldots \right]\right].
\end{align}
By Proposition \ref{p:iterated}, each operator $R_j$ defines 
a bounded operator from $H^1$ to $H^{-1}$ with the norm, 
uniformly bounded on $p\in \mathbb N^{*}$, 
$|\alpha|<c\sqrt{p}$, $y\in X$. 
Therefore, by Theorem \ref{Thm1.7W}, each operator \eqref{e:RR} defines
a bounded operator from $L^2$ to $H^{1}$ with the norm, uniformly 
bounded on $p\in \mathbb N^{*}$, $p\geq p_0$, 
$|\alpha|<c\sqrt{p}$, $y\in X$. 
This immediately completes the proof.
 \end{proof}

\subsection{Pointwise exponential estimates for the resolvents}
\label{norm-Bergman}
In this section, we derive the pointwise estimates for the Schwartz
kernel $R^{(m)}_{\lambda,p}(\cdot,\cdot)\in \mathscr{C}^{-\infty}
(X\times X, \pi_{1}^{*}(L^p\otimes E)\otimes 
\pi_{2}^{*}(L^p\otimes E)^*)$ 
of the operator $\big(\lambda-\frac{1}{p}\Delta_{p}\big)^{-m}$.
Recall that $p_{0}\in \mathbb N^{*}$, $c>0$ are given in 
Theorems \ref{Thm1.7}, \ref{Thm1.7W}.
\begin{thm}\label{p:west}
For any $m,k\in \mathbb N$ with $m>2n+k+1$, for any 
$p\in \mathbb N^{*}$, $p\geq p_0$, and $\lambda\in\delta$, we have 
$R^{(m)}_{\lambda,p}(\cdot,\cdot)\in 
\mathscr{C}^k(X\times X, \pi_1^*(L^p\otimes E)
\otimes \pi_2^*(L^p\otimes E)^*)$ 
and, for any $c_1\in (0,c)$, there exists $C_{m,k}>0$ such that
for any $p\in \mathbb N^{*}$, $p\geq p_0$, $\lambda\in\delta$, 
$x, x^\prime \in X$, we have
\begin{align}\label{e:3.43}
\big|R^{(m)}_{\lambda,p}(x, x^\prime)\big|_{C^k}
\leq C_{m,k} p^{n+\frac{k}{2}} e^{-c_1\sqrt{p} d (x, x^\prime)}.
\end{align}
\end{thm}

\begin{proof}
By \eqref{LaL}, for any $m\in \mathbb N$, $p\in \mathbb N^{*}$, 
$p\geq p_0$, $\lambda\in\delta$, $y\in X$ and $|\alpha|<c\sqrt{p}$, 
we have
\begin{align}\label{e:3.44}
f_{\alpha,p,y} \Big(\lambda-\frac{1}{p}\Delta_p\Big)^{-m}
f^{-1}_{\alpha,p,y}=\Big(\lambda-\frac{1}{p}\Delta_{p;\alpha,y}
\Big)^{-m}.
\end{align}
As in Remark \ref{r:conj}, one can show that this formula gives a
well-defined operator from $\mathscr{C}^\infty_c(X,L^p\otimes E)$ 
to $\mathscr{C}^{-\infty}(X,L^p\otimes E)$.

Since $\Delta_p$ is formally self-adjoint with respect to $\|\cdot\|_{p,0}$,
we have $\Delta_{p;\alpha,y}^*=\Delta_{p;-\alpha,y}$. Using this 
fact and Theorem \ref{Thm1.9}, we easily get that, for any 
$m_1\in \mathbb Z$, there exists $C_{m,m_1}>0$ such that, 
for all $p\in \mathbb N^{*}$, $p\geq p_0$, $\lambda\in\delta$, $y\in X$ 
and $|\alpha|<c\sqrt{p}$, we have
\begin{equation}\label{res-est}
\left\|\Big(\lambda - \frac{1}{p}\Delta_{p;\alpha,y}\Big)^{-m}
\right\|^{m_1,m_1+m}_p\leq C_{m,m_1}.
\end{equation}

The Schwartz kernel $R^{(m)}_{\lambda,p;\alpha, y}(\cdot,\cdot)\in 
\mathscr{C}^k (X\times X, \pi^*_1(L^p\otimes E)\otimes
\pi_2^*(L^p\otimes E)^*)$ of the operator
$\big(\lambda-\frac{1}{p}\Delta_{p;\alpha,y}\big)^{-m}$
is related to the Schwartz kernel $R^{(m)}_{\lambda,p}(\cdot,\cdot)$ 
of the operator $\big(\lambda-\frac{1}{p}\Delta_{p}\big)^{-m}$
by the formula
\begin{align}\label{e:3.45}
R^{(m)}_{\lambda,p;\alpha, y}(x,x^\prime)
=e^{\alpha \widetilde{d}_{p,y}(x)}R^{(m)}_{\lambda,p}(x,x^\prime) 
e^{-\alpha \widetilde{d}_{p,y}(x^\prime)}, \quad x, x^\prime \in X.
\end{align}
For $x,x^\prime,y\in X$ and $v\in (L^p\otimes E)_{x^\prime}$, 
we can write
\begin{align}\label{e:3.46}
e^{\alpha \widetilde{d}_{p,y}(x)} R^{(m)}_{\lambda,p}(x,x^\prime) 
e^{-\alpha \widetilde{d}_{p,y}(x^\prime)}v=\Big(\Big(\lambda
-\frac{1}{p}\Delta_{p;\alpha,y}\Big)^{-m}\delta_v \Big)(x)
\in (L^p\otimes E)_{x}.
\end{align}
In particular, putting $y=x^\prime$, we get
for $x,x^\prime\in X$ and $v\in (L^p\otimes E)_{x^\prime}$,
\begin{align}\label{e:3.47}
e^{\alpha \widetilde{d}_{p,x^\prime}(x)} R^{(m)}_{\lambda,p}(x,x^\prime)
e^{-\alpha \widetilde{d}_{p,x^\prime}(x^\prime)}v
=\Big(\Big(\lambda-\frac{1}{p}\Delta_{p;\alpha,x^\prime}\Big)^{-m}
\delta_v \Big)(x)\in (L^p\otimes E)_{x}.
\end{align}
By \eqref{(1.1)}, it follows that, for $0<\alpha<c\sqrt{p}$ and 
$x^\prime\in X$, we have an estimate 
$e^{\alpha \widetilde{d}_{p,x^\prime}(x^\prime)}\leq e^c$.
It is in this place that we need to use the smoothed distance function 
$\widetilde{d}$, depending on $p$. Assuming $0<\alpha<c\sqrt{p}$, 
by Propositions \ref{p:Sobolev}, \ref{p:delta}, and \eqref{e:3.47}, we get,
for $m>2n+1$, that $R^{(m)}_{\lambda,p}(x,x^\prime)$
is continuous and
\begin{equation}\label{e:3.48}
\begin{split}
\sup_{x,x^\prime\in X} e^{\alpha \widetilde{d}_{p,x^\prime}(x)}
|R^{(m)}_{\lambda,p}(x,x^\prime)|
&\leq  e^c \sup_{v\in (L^p\otimes E)_{x^\prime}, |v|=1}
\left\|\Big(\lambda-\frac{1}{p}\Delta_{p;\alpha,x^\prime}\Big)^{-m}
\delta_v \right\|_{\mathscr{C}^0_b}\\[3pt]
&\leq  C_1p^{n/2} \sup_{v\in (L^p\otimes E)_{x^\prime}, |v|=1} 
\left\|\Big(\lambda-\frac{1}{p}\Delta_{p;\alpha,x^\prime}\Big)^{-m}
\delta_v \right\|_{p,n+1}\\[3pt]
&\leq  C_2p^{n/2} \sup_{v\in (L^p\otimes E)_{x^\prime}, |v|=1}
\left\|\delta_v \right\|_{p,n+1-m}\leq C_3p^{n}.
\end{split}
\end{equation}
Similarly, for any $Q_1\in BD^{k_1}(X,L^p\otimes E)$ and 
$Q_2\in BD^{k_2}(X,L^p\otimes E)$, $k_1+k_2=k$, by the argument after \eqref{e:3.16}, we get with 
$m>2n+k+1$,
\begin{align}\begin{split}\label{e:3.49}
\sup_{x,x^\prime\in X} e^{\alpha \widetilde{d}_{p,x^\prime}(x)}
| &(Q_{1}\otimes Q_{2})R^{(m)}_{\lambda,p}(x,x^\prime)|\\
&\leq  e^c \sup_{v\in (L^p\otimes E)_{x^\prime}, |v|=1}
\left\|F_{\alpha,p,x^\prime}Q_1\Big(\lambda-\frac{1}{p}\Delta_{p}
\Big)^{-m}Q_2^*F^{-1}_{\alpha,p,x^\prime}\delta_v 
\right\|_{\mathscr{C}^0_b}\\[3pt]
&=  e^c \sup_{v\in (L^p\otimes E)_{x^\prime}, |v|=1}
\left\|Q_{1;\alpha,p,x^\prime}\Big(\lambda-\frac{1}{p}
\Delta_{p;\alpha,x^\prime}\Big)^{-m}Q_{2;\alpha,p,x^\prime}^*
\delta_v \right\|_{\mathscr{C}^0_b}\\[3pt]
&\leq  C_1p^{n/2} \sup_{v\in (L^p\otimes E)_{x^\prime}, |v|=1} 
\left\|Q_{1;\alpha,p,x^\prime}\Big(\lambda-\frac{1}{p}
\Delta_{p;\alpha,x^\prime}\Big)^{-m}Q_{2;\alpha,p,x^\prime}^*
\delta_v \right\|_{p,n+1}
\end{split}\end{align}
\begin{align}\begin{split}\nonumber
&\hspace{30mm}\leq  C_2p^{(n+k_1)/2} 
\sup_{v\in (L^p\otimes E)_{x^\prime}, |v|=1} 
\left\|\Big(\lambda-\frac{1}{p}\Delta_{p;\alpha,x^\prime}\Big)^{-m}
Q_{2;\alpha,p,x^\prime}^*\delta_v \right\|_{p,n+k_1+1}\\[3pt]
&\hspace{30mm}\leq  C_3p^{(n+k_1)/2} 
\sup_{v\in (L^p\otimes E)_{x^\prime}, |v|=1} 
\left\|Q_{2;\alpha,p,x^\prime}^*\delta_v \right\|_{p,n-m+k_1+1}\\[3pt]
&\hspace{30mm}\leq  C_4p^{(n+k)/2} 
\sup_{v\in (L^p\otimes E)_{x^\prime}, |v|=1} 
\left\|\delta_v \right\|_{p,n-m+k+1}\leq C_5p^{n+k/2}.
\end{split}\end{align}

For any $x_0\in X$, fix an orthonormal basis $e_1,\ldots,e_{2n}$ in 
$T_{x_0}X$. As above, we extend it to 
a frame $e_1,\ldots,e_{2n}$ on $B^{X}(0,\varepsilon)$ 
as constant vector fields on $T_{x_{0}}X$.
One can show (see, for instance, \cite[Proposition 1.5]{Kor91}) 
that the norm $\|\cdot\|_{\mathscr{C}^k_b}$ on
$\mathscr{C}^k_b(X, L^p\otimes E)$ is equivalent uniformly on 
$p\in \mathbb N^{*}$ to the norm $\|\cdot\|^\prime_{\mathscr{C}^k_b}$ 
given for $ u\in \mathscr{C}^k_b(X, L^p\otimes E)$ by
\begin{align}\label{e:3.50}
\|u\|^\prime_{\mathscr{C}^k_b}=\sup_{x_0\in X}\sup_{0\leq \ell \leq k}
\left\| \nabla^{L^p\otimes E}_{e_{j_1}}\cdots 
\nabla^{L^p\otimes E}_{e_{j_\ell}}u(x_0)\right\|.
\end{align}
That is, there exists $C_k>0$ such that, for any $p\in \mathbb N^{*}$ 
and $u\in \mathscr{C}^k_b(X, L^p\otimes E)$, we have
\begin{equation}\label{Ckb-prime}
C_k^{-1}\|u\|_{\mathscr{C}^k_b}\leq \|u\|^\prime_{\mathscr{C}^k_b}
\leq C_k\|u\|_{\mathscr{C}^k_b}.
\end{equation}

Let $\phi\in \mathscr{C}^\infty_c(\mathbb R^{2n})$ be any function 
supported in the ball $B(0,\varepsilon)$ such that $\phi\equiv 1$ on 
$B(0,\varepsilon/2)$. Consider the function 
$\phi_{x_0}\in \mathscr{C}^\infty_c(B^{X}(x_0,\varepsilon))$, 
corresponding to $\phi$ under the isomorphisms 
$\mathscr{C}^\infty_c(B(0,\varepsilon))\cong 
\mathscr{C}^\infty_c(B^{T_{x_0}X}(0,\varepsilon))
\cong \mathscr{C}^\infty_c(B^{X}(x_0,\varepsilon))$
induced by the basis $e_1,\ldots,e_{2n}$ and the exponential map 
$\exp^X_{x_0}$. The family $\{\phi_{x_0}, x_0\in X\}$ is bounded in 
$\mathscr{C}^\infty_b(X)$.

By \eqref{Ckb-prime}, it follows that there exists $C_k>0$ such that,
for any $x,x^\prime\in X$, we have
\begin{align}\label{e:3.51}
|R^{(m)}_{\lambda,p}(x,x^\prime)|_{\mathscr{C}^k}
\leq  C_k\sup_{Q_1,Q_2} |(Q_{1}\otimes Q_{2})
R^{(m)}_{\lambda,p}(x,x^\prime)|,
\end{align}
with the supremum taken over all pairs $(Q_1,Q_2)$, where 
$Q_1\in BD^{k_1}(X,L^p\otimes E)$ and 
$Q_2\in BD^{k_2}(X,L^p\otimes E)$, $k_1+k_2\leq k$, have the form
\begin{align}\label{e:3.52}
Q_1=\nabla^{L^p\otimes E}_{e_{i_1}}\cdots 
\nabla^{L^p\otimes E}_{e_{i_{k_1}}}\phi_x,\quad
Q_2=\nabla^{L^p\otimes E}_{e_{j_1}}\cdots 
\nabla^{L^p\otimes E}_{e_{j_{k_2}}}\phi_{x^\prime},
\end{align}
and $i_1,\ldots, i_{k_1},j_1,\ldots, j_{k_2}\in \{1,\ldots,2n\}$, 
that immediately completes the proof of \eqref{e:3.43}.
\end{proof}

\subsection{Proof of Theorems \ref{t:mainPp} and~\ref{t:covering}}
In this section, we complete the proofs of Theorems \ref{t:mainPp} 
and~\ref{t:covering}.

\begin{proof}[Proof of Theorem \ref{t:mainPp}]
Let $k\in \mathbb N$. Take an arbitrary $m>2n+k+1$. 
By \eqref{bergman-integral}, we have
\begin{equation}\label{bergman-integral1}
P_{p}(x,x^\prime)=\frac{1}{2\pi i}
\int_\delta \lambda^{m-1}R^{(m)}_{\lambda,p}(x, x^\prime)\,d\lambda.
\end{equation}
By Theorem~\ref{p:west}, it implies immediately the 
$\mathscr{C}^k$-estimate in Theorem \ref{t:mainPp} with $c>0$ 
given by that theorem.
\end{proof}

\begin{proof}[Proof of Theorem \ref{t:covering}] First, we observe that, 
	for $\widetilde \tau\in \mathscr{C}^\infty(\widetilde X)$ and
	$\tau\in \mathscr{C}^\infty(X)$ given by \eqref{tau}, we have 
	$\widetilde \tau=\pi^*\tau$. Next, the quantities $\widetilde\mu_0$ 
	and $\mu_0$ defined by \eqref{mu} are equal. In particular, 
	$\widetilde\mu_0>0$.
By Theorem~\ref{t:gap0}, there exists constant $C_L>0$ such that
for any $p\in \mathbb N^{*}$
 \begin{align}\label{e:3.55}
 \sigma(\Delta_p)\cup \sigma(\widetilde \Delta_p)\subset [-C_L,C_L]
 \cup [2p\mu_0-C_L,+\infty).
 \end{align}
Recall that $\delta$ denotes the counterclockwise oriented circle in 
$\mathbb C$ centered at $0$ of radius $\mu_0$.

For any $p,m\in \mathbb N$, $p\geq p_0$, and $\lambda\in\delta$, 
denote by $\widetilde R^{(m)}_{\lambda,p}\in 
\mathscr{C}^{-\infty}(\widetilde X\times \widetilde X, \pi_1^*
(\widetilde L^p\otimes \widetilde E)\otimes
\pi_2^*(\widetilde L^p\otimes \widetilde E)^*)$ 
and $R^{(m)}_{\lambda,p}\in \mathscr{C}^{-\infty}(X\times X, 
\pi_1^*(L^p\otimes E)\otimes \pi_2^*(L^p\otimes E)^*)$
the Schwartz kernels  of the operators 
$\big(\lambda-\frac{1}{p}\widetilde \Delta_{p}\big)^{-m}$ 
and $\big(\lambda-\frac{1}{p}\Delta_{p}\big)^{-m}$, respectively.
Recall that, for any $x^\prime$, they satisfy the identities
\begin{equation}\label{delta-eq0}
\Big(\lambda-\frac{1}{p}\widetilde \Delta_{p}\Big)^{m}
\widetilde R^{(m)}_{\lambda,p}(\cdot,x^\prime)
=\delta_{x^\prime},\quad \Big(\lambda-\frac{1}{p}\Delta_{p}\Big)^{m}
R^{(m)}_{\lambda,p}(\cdot,x^\prime)=\delta_{x^\prime}.
\end{equation}
Moreover, $u=R^{(m)}_{\lambda,p}(\cdot,x^\prime)$ is the unique 
distributional solution of the equation
\begin{equation}\label{delta-eq}
 \Big(\lambda-\frac{1}{p}\Delta_{p}\Big)^{m}u=\delta_{x^\prime}.
\end{equation}

Let $m>2n+1$. Then, by elliptic regularity and Sobolev embedding theorem, 
$\widetilde R^{(m)}_{\lambda,p}$ and $R^{(m)}_{\lambda,p}$ 
are continuous. 
We claim that there exists $p_1\in\mathbb N$ such that for any 
$p>p_1$ and $x,x^\prime\in \widetilde X$,
\begin{equation}\label{av-res}
\sum_{\gamma\in \Gamma}\widetilde
R^{(m)}_{\lambda,p}(\gamma x,x^\prime) 
= R^{(m)}_{\lambda,p}(\pi(x), \pi(x^\prime)).
\end{equation}

By \cite{Milnor}, there exists $K>0$ such that
$\sum_{\gamma\in \Gamma}e^{-ad(\gamma x, x^\prime)}<+\infty$
for any 
$a>K$ and $x,x^\prime\in \widetilde X$. Put $p_1\geq K^2/c^2+p_0$. 
Then, by Theorem~\ref{p:west}, for any $p>p_1$ and
$x,x^\prime\in \widetilde X$, the series in 
the left-hand side of \eqref{av-res} 
is absolutely convergent with respect to $\mathscr{C}^0$-norm 
and its sum
\begin{align}\label{e:3.56}
\widetilde S^{(m)}_{\lambda,p}(x,x^\prime)
=\sum_{\gamma\in \Gamma}\widetilde R^{(m)}_{\lambda,p}
(\gamma x,x^\prime)
\end{align}
is a $\Gamma$-invariant continuous section on 
$\widetilde X\times \widetilde X$. 
So we can write 
\begin{align}\label{e:3.57}
	\widetilde S^{(m)}_{\lambda,p}(x,x^\prime)
=S^{(m)}_{\lambda,p}(\pi(x), \pi(x^\prime))
\text{   with } S^{(m)}_{\lambda,p}\in \mathscr{C}^0
(X\times X, \pi_1^*(L^p\otimes E)\otimes \pi_2^*(L^p\otimes E)^*).
\end{align}

Moreover, by \eqref{delta-eq0}, for any $x^\prime \in \widetilde X$, 
$\widetilde S^{(m)}_{\lambda,p}(\cdot,x^\prime)$ satisfies the identity
\begin{equation}\label{e:3.58}
\Big(\lambda-\frac{1}{p}\widetilde \Delta_{p}\Big)^{m} \,
\widetilde S^{(m)}_{\lambda,p}(\cdot,x^\prime)=\delta_{x^\prime}.
\end{equation}
Therefore, for any $x^\prime\in X$, 
$S^{(m)}_{\lambda,p}(\cdot,x^\prime)$ satisfies the identity
\begin{equation}\label{e:3.59}
\Big(\lambda-\frac{1}{p}\Delta_{p}\Big)^{m}\,
S^{(m)}_{\lambda,p}(\cdot,x^\prime)=\delta_{x^\prime}.
\end{equation}
By the uniqueness of the solution of \eqref{delta-eq}, it follows that 
$S^{(m)}_{\lambda,p}=R^{(m)}_{\lambda,p}$. This completes the 
proof of \eqref{av-res}.

Since, by Theorem~\ref{p:west}, the series in the left-hand side of 
\eqref{av-res} is absolutely convergent with respect to the
$\mathscr{C}^0$-norm uniformly in $\lambda\in \delta$, we can integrate
it term by term over $\delta$. Using \eqref{bergman-integral1} and 
\eqref{av-res},
for any $p>p_1$ and $x,x^\prime\in \widetilde X$, we get
\begin{equation}\label{e:3.61}
\begin{split}
\sum_{\gamma\in \Gamma}\widetilde P_p(\gamma x,x^\prime) 
&=\frac{1}{2\pi i}\sum_{\gamma\in \Gamma} \int_\delta \lambda^{m-1} 
\widetilde R^{(m)}_{\lambda,p}(\gamma x, x^\prime)\,d\lambda\\ 
&= \frac{1}{2\pi i}
\int_\delta \lambda^{m-1}R^{(m)}_{\lambda,p}(\pi(x), \pi(x^\prime))\,
d\lambda= P_p(\pi(x), \pi(x^\prime)).
\end{split}
\end{equation}
The proof of Theorem \ref{t:covering} is completed.
\end{proof}

\section{Full off-diagonal asymptotic expansion} \label{expansions}
In this section, we study the full off-diagonal asymptotic expansion
in the geometric situation of our article described in the Introduction.

In the case of a compact K\"ahler manifold the asymptotic 
expansion of the Bergman kernel 
$P_p(x,x)$ restricted to the diagonal was 
initiated by Tian \cite{Tian}, who proved the expansion up to first order.
Catlin \cite{Catlin} and Zelditch \cite{Zelditch98} proved the asymptotic
expansion of $P_p(x,x)$ up to arbitrary order, see \cite{MM07} for the 
numerous applications of these results.  

On the other hand, the off-diagonal expansion of the Bergman kernel has
many applications.
In the case of complex manifolds the expansion of $P_p(x,x')$
holds in a fixed neighborhood of the diagonal (independent of $p$),
see \cite{DLM04a}, \cite[Theorem 4.2.1]{MM07}. Such kind of expansion
is called full off-diagonal expansion.
As already noted in \cite[Problem 6.1, p.\ 292]{MM07}
the proof of the full off-diagonal expansion holds also for
complex manifolds with bounded geometry.

In the case of the Bergman kernel associated to the
renormalized Bochner-Laplacian considered
in the present article, it was shown
in \cite[Theorem 1.19]{ma-ma08} that the off-diagonal expansion holds in
a neighborhood of size $1/\sqrt{p}$ of the diagonal. This is called
near off-diagonal expansion. 
Moreover, it was shown in \cite[p.\,329]{MM07}
that the Bergman kernel is $\mathcal{O}(p^{-\infty})$ outside 
a neighborhood of size $p^{-\theta}$, for any $\theta\in(0,1/2)$.
These estimates are used in the proof of the Kodaira embedding theorem
for symplectic manifolds \cite[Theorem 3.6]{ma-ma08}.
In \cite{lu-ma-ma} a less precise estimate than
\cite[Theorem 1.19]{ma-ma08}
was obtained in a neighborhood
of size $p^{-\theta}$, $\theta\in(0,1/2)$, which is however enough to
derive the Berezin-Toeplitz quantization for the quantum spaced $\mH_p$
in \cite{ioos-lu-ma-ma}.
Finally, in \cite{bergman} the full off-diagonal expansion
for the Bergman kernel associated to the renormalized
Bochner-Laplacian was proved by combining \cite[\S 1]{ma-ma08}
and weight function trick in \cite{Kor91}.

In this section we extend the results on the full off-diagonal asymptotic 
expansion to the case of manifolds of bounded geometry. Moreover, 
using the technique of weighted estimates of the previous sections, 
we slightly improve the remainder in the asymptotic expansions. 
We will keep the setting described in Introduction.

Consider the fiberwise product $TX\times_X TX=\{(Z,Z^\prime)
\in T_{x_0}X\times T_{x_0}X : x_0\in X\}$. Let 
$\pi : TX\times_X TX\to X$ be the natural projection given by  
$\pi(Z,Z^\prime)=x_0$. The kernel $P_{q,p}(x,x^\prime)$
(of the operator $(\Delta_{p})^{q} P_{\mathcal{H}_{p}}$) induces 
a smooth section $P_{q,p,x_0}(Z,Z^\prime)$ of the vector bundle 
$\pi^*(\operatorname{End}(E))$ on $TX\times_X TX$ defined for all
$x_0\in X$ and $Z,Z^\prime\in T_{x_0}X$ with $|Z|, |Z^\prime|<a^X$.

We will follow the arguments of \cite{bergman} and \cite{ma-ma08}. 
We will use the normal coordinates near an arbitrary point $x_0\in X$
introduced in the proof of Theorem~\ref{Thm1.9}. Let $\{e_j\}$ be 
an oriented orthonormal basis of $T_{x_0}X$. It gives rise to an
isomorphism $X_0:=\mathbb R^{2n}\cong T_{x_0}X$.  

Consider the trivial bundles $L_0$ and $E_0$ with fibers $L_{x_0}$ 
and $E_{x_0}$ on $X_0$. Recall that we have the Riemannian metric 
$g^{TX}$ on $B^{T_{x_0}X}(0,a^X)$ as well as the connections 
$\nabla^L$, $\nabla^E$ and the Hermitian metrics $h^L$, $h^E$ 
on the restrictions of $L_0$, $E_0$ to $B^{T_{x_0}X}(0,a^X)$
induced by the identification $B^{T_{x_0}X}(0,a^X)\cong B^{X}(x_0,a^X)$.
In particular, $h^L$, $h^E$ are the constant metrics 
$h^{L_0}=h^{L_{x_0}}$, $h^{E_0}=h^{E_{x_0}}$. 
For $\varepsilon \in (0,a^X/4)$, one can extend these geometric objects
from $B^{T_{x_0}X}(0,\varepsilon)$ to $X_0\cong T_{x_0}X$
in the following way.

Let $\rho : \mathbb R\to [0,1]$ be a smooth even function such
that $\rho(v)=1$ if $|v|<2$ and $\rho(v)=0$ if $|v|>4$. 
Let $\varphi_\varepsilon : \mathbb R^{2n}\to \mathbb R^{2n}$ 
be the map defined by $\varphi_\varepsilon(Z)=\rho(|Z|/\varepsilon)Z$. 
We equip $X_0$ with the metric $g^{TX_0}(Z)=g(\varphi_\varepsilon(Z))$. 
Set $\nabla^{E_0}=\varphi^*_\varepsilon\nabla^E$. Define  the Hermitian 
connection $\nabla^{L_0}$ on $(L_0,h^{L_0})$
by (cf.\  \cite[(4.23)]{DLM04a}, \cite[(1.21)]{ma-ma08})
\begin{align}\label{e:4.1}
\nabla^{L_0}\left|_Z\right.=\varphi_\varepsilon^*\nabla^L
+\frac 12(1-\rho^2(|Z|/\varepsilon))R^L_{x_0}(\mathcal R,\cdot),
\end{align}
where $\mathcal R(Z)=\sum_j Z_je_j\in \mathbb R^{2n}\cong T_ZX_0$.

By \cite[(4.24)]{DLM04a}, \cite[(1.22)]{ma-ma08},
if $\varepsilon$ is small enough, then the curvature $R^{L_0}$
is positive and satisfies the following estimate for any $x_0\in X$,
\begin{align}\label{e:4.2}
\inf_{u\in T_{x_0}X\setminus\{0\}}\frac{iR^{L_0}(u,J^{L_0}u)}
{|u|_{g^TX}^2}\geq \frac{4}{5}\mu_0.
\end{align}

Here $J^{L_0}$ is the almost complex structure on $X_0$ defined by 
$g^{TX_0}$ and $iR^{L_0}$. From now on, we fix such an $\varepsilon>0$.

Let $dv_{TX}$ be the Riemannian volume form of 
$(T_{x_0}X, g^{T_{x_0}X})$. Let $\kappa$ be the smooth positive
function on $X_0$ defined by the equation
\begin{equation}\label{e:kappa}
dv_{X_0}(Z)=\kappa(Z)dv_{TX}(Z), \quad Z\in X_0.
\end{equation}

Let $\Delta_p^{X_0}=\Delta^{L_0^p\otimes E_0}-p\tau_0$ be
the associated renormalized Bochner-Laplacian acting on 
$\mathscr{C}^\infty(X_0,L_0^p\otimes E_0)$. Then
(cf.\  \cite[(1.23)]{ma-ma08}) there exists $C_{L_0}>0$
such that for any $p$
 \begin{equation}\label{gap0}
 \sigma(\Delta_p^{X_0})\subset [-C_{L_0},C_{L_0}]\cup 
 \left[\frac{8}{5}p\mu_0-C_{L_0},+\infty\right).
 \end{equation}

Consider the subspace $\mathcal H^0_p$ in
$\mathscr{C}^\infty(X_0,L_0^p\otimes E_0)\cong
\mathscr{C}^\infty(\mathbb R^{2n}, E_{x_0})$
spanned by the eigensections of $\Delta^{X_0}_p$
corresponding to eigenvalues in
$[-C_{L_0},C_{L_0}]$. Let $P_{\mathcal H^0_p}$
be the orthogonal projection onto
$\mathcal H^0_p$. The smooth kernel of
$(\Delta^{X_0}_p)^qP_{\mathcal H^0_p}$
with respect to the Riemannian volume form $dv_{X_0}$ is denoted by
$P^0_{q,p}(Z,Z^\prime)$. As proved in \cite[Proposition 1.3]{ma-ma08},
the kernels $P_{q,p,\,x_0}(Z,Z^\prime)$ and $P^0_{q,p}(Z,Z^\prime)$
are asymptotically close on $B^{T_{x_0}X}(0,\varepsilon)$ in the
$\mathscr{C}^\infty$-topology, as $p\to \infty$. 

In the next theorem,
we improve the $\mathcal O(p^{-\infty})$-estimate for the norms 
of the difference between $P_{q,p,x_0}(Z,Z^\prime)$ and
$P^0_{q,p}(Z,Z^\prime)$ of \cite[Proposition 1.3]{ma-ma08}, 
proving an exponential decay estimate. This is essentially the main 
new result of this section.

\begin{thm}\label{prop13}
There exists $c_0>0$ such that, for any $k\in \mathbb N$,
there exists $C_{k}>0$ such that for $p\in \mathbb N^{*}$, $x_0\in X$
and $Z,Z^\prime\in B^{X_0}(0,\varepsilon)$,
 \begin{align}\label{e:4.5}
\big|P_{q,p,x_0}(Z,Z^\prime)-P^0_{q,p}(Z,Z^\prime)
\big|_{\mathscr{C}^k}\leq   C_{k}e^{-c_0\sqrt{p}}.
\end{align}
 \end{thm}

\begin{proof}
As in Section~\ref{norm-Bergman}, $R^{(m)}_{\lambda,p}\in 
\mathscr{C}^{-\infty}(X\times X, \pi_1^*(L^p\otimes E)\otimes 
\pi_2^*(L^p\otimes E)^*)$ denotes the Schwartz kernel of the operator
$\Big(\lambda-\frac{1}{p}\Delta_{p}\Big)^{-m}$. We also denote by 
\begin{align}\label{e:4.6}
	R^{X_0,(m)}_{\lambda,p}\in \mathscr{C}^{-\infty}(X_0\times X_0,
\pi_1^*(L_0^p\otimes E_0)\otimes \pi_2^*(L^p_0\otimes E_0)^*)
\end{align}
the Schwartz kernel of the operator $\Big(\lambda-\frac{1}{p}
\Delta^{X_0}_{p}\Big)^{-m}$. For $m>2n+1$, the distributional sections 
$R^{(m)}_{\lambda,p}(x,x^\prime)$ and 
$R^{X_0,(m)}_{\lambda,p}(Z,Z^\prime)$ are continuous sections.

Recall that $d$ denotes the geodesic distance on $X$ and $d^{X_0}$ 
the geodesic distance on $X_0$. By construction, $d$ coincides with 
$d^{X_0}$ on $B^X(x_0, 2\varepsilon)\times B^{X}(x_0, 2\varepsilon)
\cong B^{X_0}(0, 2\varepsilon)\times B^{X_0}(0, 2\varepsilon)$. 
Let $\psi \in \mathscr{C}^\infty_c(B^{X_0}(0,4\varepsilon))$, 
$\psi\equiv 1$ on $B^{X_0}(0,3\varepsilon)$.

Consider a function $\chi\in \mathscr{C}^\infty(\mathbb R)$ such that 
$\chi(r)=1$ for $|r|\leq \varepsilon$  and $\chi(r)=0$ for 
$|r|\geq 2\varepsilon$.
For any $\alpha\in \mathbb R$ and $W\in X_0$, we introduce
a weight function $\phi^{X_0}_{\alpha,W} \in \mathscr{C}^{\infty}(X_0)$
by
\begin{align}\label{e:4.7}
\phi^{X_0}_{\alpha,W}(Z) = \exp \left[{\alpha 
\chi(d ^{X_0}(Z,W))}\right],\quad Z \in X_0.
\end{align}

Consider the operators
\begin{align}\label{e:4.8}
\Delta^{X_0}_{p;\alpha,W}:=\phi^{X_0}_{\alpha,W}
\Delta^{X_0}_{p}(\phi^{X_0}_{\alpha,W})^{-1}.
\end{align}

With the same arguments as in Theorem \ref{Thm1.9}, 
one can show that

\begin{thm}\label{Thm1.9a}
There exist $c_0>0$ and $p_0\in \mathbb N^{*}$ such that, for any 
$p\in \mathbb N^{*}$, $p\geq p_0$, $\lambda\in \delta$, $W\in X_0$ 
and $|\alpha|<c_0\sqrt{p}$,
the operator $\lambda-\frac{1}{p}\Delta^{X_0}_{p;\alpha,W}$
is invertible in $L^2$. Moreover,  for any $m\in \mathbb N$, 
the resolvent 
$\big(\lambda-\frac{1}{p}\Delta^{X_0}_{p;\alpha,W}\big)^{-1}$ 
maps $H^m$ to $H^{m+1}$ with the following norm estimates:
\begin{equation}\label{e:mm+1-again}
\left\|\Big(\lambda-\frac{1}{p}\Delta^{X_0}_{p;\alpha,W}\Big)^{-1}
\right\|_{p}^{m,m+1}\leq C,
\end{equation}
where $C>0$ is independent of $p\in \mathbb N^{*}$, $p\geq p_0$, 
$\lambda\in \delta$, $W\in X_0$ and $x_0\in X$.
\end{thm}

For any $Z,W\in B^{X_0}(0,\varepsilon/2)$, we have 
$d^{X_0}(Z,W)<\varepsilon$ and, therefore, 
$\phi^{X_0}_{\alpha,W}(Z) =e^{\alpha}$.
For $m>2n+1$, we have
\begin{equation}\label{e:diff}
\begin{split}
&\sup_{Z,W\in B^{X_0}(0,\varepsilon/2)}  |R^{(m)}_{\lambda,p}(Z,W)
-R^{X_0, (m)}_{\lambda,p}(Z,W)| \\
&\leq e^{-\frac 12c_0\sqrt{p}}\sup_{Z,W\in B^{X_0}(0,\varepsilon/2)} 
\left|\phi_{\frac 12c_0\sqrt{p},W}^{X_0}(Z) \Big(\psi(Z)
R^{(m)}_{\lambda,p}(Z,W)-R^{X_0, (m)}_{\lambda,p}(Z,W)\psi(W)\Big)
\right|\\
&\leq e^{-\frac 12c_0\sqrt{p}} 
\sup_{\substack{W\in B^{X_0}(0,\varepsilon/2),\\
v\in (L_0^p\otimes E_0)_W, |v|=1}} 
\left\|\phi_{\frac 12c_0\sqrt{p},W}^{X_0} 
\Big(\psi \Big(\lambda-\frac{1}{p}\Delta_{p}\Big)^{-m}
-\Big(\lambda-\frac{1}{p}\Delta^{X_0}_{p}\Big)^{-m}\psi\Big)
\delta_v \right\|_{\mathscr{C}^0_b}\\ 
&\leq Cp^{n/2} 
e^{-\frac 12c_0\sqrt{p}}
\sup_{\substack{W\in B^{X_0}(0,\varepsilon/2),\\ 
v\in (L_0^p\otimes E_0)_W, |v|=1}} 
\left\|\phi_{\frac 12c_0\sqrt{p},W}^{X_0}
\Big(\psi \Big(\lambda-\frac{1}{p}\Delta_{p}\Big)^{-m}
-\Big(\lambda-\frac{1}{p}\Delta^{X_0}_{p}\Big)^{-m}\psi\Big)
\delta_v \right\|_{p,n+1}.
\end{split}
\end{equation}
By construction, we have for any 
$u \in \mathscr{C}^\infty_c(B^{X_0}(0,2\varepsilon),E_{x_0})$,
\begin{align}\label{e:4.10}
\Delta_pu(Z)=\Delta^{X_0}_pu(Z)
\end{align}
Then, for any $u$, we have
\begin{equation}\label{e:4.11}
\Big(\lambda-\frac{1}{p}\Delta^{X_0}_p\Big)^{m}
\psi\Big(\lambda-\frac{1}{p}\Delta_p\Big)^{-m}u =
\left[\Big(\lambda-\frac{1}{p}\Delta^{X_0}_p\Big)^{m},
\psi\right]\Big(\lambda-\frac{1}{p}\Delta_p\Big)^{-m}u+\psi u
\end{equation}
and
\begin{multline}\label{e:4.12}
\psi\Big(\lambda-\frac{1}{p}\Delta_p\Big)^{-m}u
- \Big(\lambda-\frac{1}{p}\Delta^{X_0}_p\Big)^{-m}\psi u\\
=\Big(\lambda-\frac{1}{p}\Delta^{X_0}_p\Big)^{-m}
\left[\Big(\lambda-\frac{1}{p}\Delta^{X_0}_p\Big)^{m}, 
\psi\right]\Big(\lambda-\frac{1}{p}\Delta_p\Big)^{-m}u.
\end{multline}

Now, for any $p\in \mathbb N^{*}$, $\lambda\in \delta$ and $W\in X_0$, 
by Theorem~\ref{Thm1.9a} and \eqref{e:4.12}, we have for $m>2n+1$,
$v\in (L_0^p\otimes E_0)_W, |v|=1$,
\begin{multline}\label{e:4.13}
\left\|\phi^{X_0}_{\frac 12c_0\sqrt{p},W} 
\Big(\psi\Big(\lambda-\frac{1}{p}\Delta_p\Big)^{-m}
- \Big(\lambda-\frac{1}{p}\Delta^{X_0}_p\Big)^{-m}\psi\Big) 
\delta_v\right\|_{p,n+1}\\ 
\leq C \left\|\phi^{X_0}_{\frac 12c_0\sqrt{p},W} 
\left[\Big(\lambda-\frac{1}{p}\Delta^{X_0}_p\Big)^{m}, 
\psi\right]\Big(\lambda-\frac{1}{p}\Delta_p\Big)^{-m}
\delta_v\right\|_{p,n+1-m}.
\end{multline}
Since the operator 
$\left[\Big(\lambda-\frac{1}{p}\Delta^{X_0}_p\Big)^{m}, \psi\right]$ 
vanishes on $B^{X_0}(0,3\varepsilon)$, for any 
$W\in B^{X_0}(0,\varepsilon)$ we have
$d^{X_0}(W,Z)>2\varepsilon$ and, therefore, 
$\phi^{X_0}_{\frac 12c_0\sqrt{p},W}(Z)=1$ on the support of 
$\left[\Big(\lambda-\frac{1}{p}\Delta^{X_0}_p\Big)^{m}, \psi\right]$. 
Hence, for $W\in B^{X_0}(0,\varepsilon)$, by \eqref{e:4.13}, we get
\begin{equation}\label{e:4.14}
\begin{split}
&\left\|\phi^{X_0}_{\frac 12c_0\sqrt{p},W} 
\Big(\psi\Big(\lambda-\frac{1}{p}\Delta_p\Big)^{-m} 
- \Big(\lambda-\frac{1}{p}\Delta^{X_0}_p\Big)^{-m}\psi\Big)
\delta_v\right\|_{p,n+1}\\[3pt] 
&\leq C \left\| 
\left[\Big(\lambda-\frac{1}{p}\Delta^{X_0}_p\Big)^{m}, 
\psi\right]\!\Big(\lambda-\frac{1}{p}\Delta_p\Big)^{-m}\delta_v
\right\|_{p,n+1-m}\\[3pt]
&\leq C \left\|\Big(\lambda
-\frac{1}{p}\Delta_p\Big)^{-m}\delta_v\right\|_{p,n+1} 
\leq C \left\|\delta_v\right\|_{p,n+1-m}\leq C p^{n/2}.
\end{split}
\end{equation}
and, finally, by \eqref{e:diff} and \eqref{e:4.14}, we get
\begin{equation}\label{e:4.15}
 \sup_{Z,W\in B^{X_0}(0,\varepsilon/2)}  \big|R^{(m)}_{\lambda,p}(Z,W)
 -R^{X_0, (m)}_{\lambda,p}(Z,W)\big|  \leq C_1p^{n} 
 e^{-\frac 12c_0\sqrt{p}} \leq C_2 e^{-\frac 14c_0\sqrt{p}}.
\end{equation}
Using \eqref{bergman-integral1}, we complete the proof for $k=0$.
The proof of the case of arbitrary $k$ can be given similarly to
the proof of Theorem \ref{p:west}.
\end{proof}

Now we can proceed as in \cite{bergman,ma-ma08}. We only observe 
that all the constants in the estimates in \cite{bergman,ma-ma08} depend 
on finitely many derivatives of $g^{TX}$, $h^L$, $\nabla^L$, $h^E$, 
$\nabla^E$, $J$ and the lower bound of $g^{TX}$. Therefore, by
the bounded geometry assumptions, all the estimates are uniform on
the parameter $x_0\in X$. We will omit the details and give only the
final result, stating the full off-diagonal asymptotic expansion of the 
generalized Bergman kernel $P_{q,p}$ as $p\to \infty$ 
(see Theorem~\ref{t:main} below).

The almost complex structure $J_{x_0}$ induces a decomposition 
$T_{x_0}X\otimes_{\mathbb R}\mathbb
C=T^{(1,0)}_{x_0}X\oplus T^{(0,1)}_{x_0}X$, where 
$T^{(1,0)}_{x_0}X$ and $T^{(0,1)}_{x_0}X$ are the eigenspaces 
of $J_{x_0}$ corresponding to eigenvalues $i$ and $-i$ respectively. 
Denote by $\det_{\mathbb C}$ the determinant function of the complex
space $T^{(1,0)}_{x_0}X$. Put
\begin{align}\label{e:4.17}
\mathcal J_{x_0}=-2\pi i J_0.
\end{align}
Then $\mathcal J_{x_0} : T^{(1,0)}_{x_0}X\to T^{(1,0)}_{x_0}X$
is positive, and $\mathcal J_{x_0} : T_{x_0}X\to T_{x_0}X$ is
skew-adjoint \cite[(1.81)]{ma-ma08} (cf.\  \cite[(4.114)]{DLM04a}).
We define a function $\mathcal P=\mathcal P_{x_0}\in 
\mathscr{C}^\infty(T_{x_0}X\times T_{x_0}X)$ by
\begin{equation}\label{e:Bergman}
\mathcal P(Z, Z^\prime)
=\frac{\det_{\mathbb C}\mathcal J_{x_0}}{(2\pi)^n}
\exp\Big(-\frac 14\langle (\mathcal J^2_{x_0})^{1/2}(Z-Z^\prime), 
(Z-Z^\prime)\rangle +\frac 12 \langle \mathcal J_{x_0} Z, Z^\prime 
\rangle \Big).
\end{equation}
It is the Bergman kernel of the second order differential operator 
$\mathcal L_0$ on $\mathscr{C}^\infty(T_{x_0}X,\mathbb{C})$ given by
\begin{equation}\label{e:L0}
\mathcal L_0=-\sum_{j=1}^{2n} \Big(\nabla_{e_j}
+\frac 12 R^L_{x_0}(Z,e_j)\Big)^2-\tau (x_0),
\end{equation}
where $\{e_j\}_{j=1,\ldots,2n}$ is an orthonormal base in $T_{x_0}X$.
Here, for $U\in T_{x_0}X$, we denote by $\nabla_U$ the ordinary operator 
of differentiation in the direction $U$ on the space 
$\mathscr{C}^\infty(T_{x_0}X, \mathbb{C})$. Thus, $\mathcal P$ 
is the smooth kernel (with respect to $dv_{TX}$) of the orthogonal 
projection in $L^2(T_{x_0}X, \mathbb{C})$ 
to the kernel of $\mathcal L_0$.

Let $\kappa$ be the function defined in \eqref{e:kappa}.

\begin{thm}\label{t:main}
There exists $\varepsilon\in (0,a^X)$ such that, for any
$j,m,m^\prime\in \mathbb N$, $j\geq 2q$, there exist positive 
constants $C$, $c$ and $M$ such that for any $p\geq 1$, $x_0\in X$
and $Z,Z^\prime\in T_{x_0}X$, $|Z|, |Z^\prime|<\varepsilon$, we have
\begin{multline}\label{e:main}
\sup_{|\alpha|+|\alpha^\prime|\leq m}
\Bigg|\frac{\partial^{|\alpha|+|\alpha^\prime|}}
{\partial Z^\alpha\partial Z^{\prime\alpha^\prime}}
\Bigg(\frac{1}{p^n}P_{q,p,x_0}(Z,Z^\prime)
-\sum_{r=2q}^jF_{q,r,x_0}(\sqrt{p} Z, \sqrt{p}Z^\prime)
\kappa^{-\frac 12}(Z)\kappa^{-\frac 12}(Z^\prime)p^{-\frac{r}{2}+q}
\Bigg)\Bigg|_{\mathscr{C}^{m^\prime}(X)}\\
\leq Cp^{-\frac{j-m+1}{2}+q}(1+\sqrt{p}|Z|+\sqrt{p}|Z^\prime|)^M
\exp(-c\sqrt{\mu_0p}|Z-Z^\prime|)+\mathcal O(e^{-c_0\sqrt{p}}),
\end{multline}
where
\begin{align}\label{e:4.19}
F_{q,r,x_0}(Z,Z^\prime)=J_{q,r,x_0}(Z,Z^\prime)\mathcal 
P_{x_0}(Z,Z^\prime),
\end{align}
$J_{q,r,x_0}(Z,Z^\prime)$ are polynomials in $Z, Z^\prime$, 
depending smoothly on $x_0$, with the same parity as $r$ and 
$\operatorname{deg} J_{q,r,x_0}\leq 3r$.
\end{thm}

Here $\mathscr{C}^{m^\prime}(X)$ is the 
$\mathscr{C}^{m^\prime}$-norm for the parameter $x_0\in X$. 
Note that the summation in \eqref{e:main} starts from $r=2q$
and \eqref{e:4.19} due
to \cite[Th 1.18]{ma-ma08}.

\section{Berezin-Toeplitz quantization on orbifolds} \label{pbs4}

After the pioneering work of Berezin, the Berezin-Toeplitz 
quantization achieved generality for compact 
K\"ahler manifolds and trivial bundle $E$
through the works \cite{BouGou81,BouSt76, BMS94}. 
We refer to \cite{ma:ICMtalk,MM07,ma-ma08,MM11} 
for more references and background.
The theory of Berezin-Toeplitz quantization on K\"ahler
and symplectic orbifolds was first  established by
Ma and Marinescu \cite[Theorems 6.14, 6.16]{ma-ma08a}
by using as quantum spaces the kernel of the spin$^c$ Dirac operator.
Especially,
they showed that the set of Toeplitz operators forms an algebra.
The main tool was the asymptotic expansion of the Bergman kernel 
associated with the spin$^c$ Dirac operator of Dai-Liu-Ma \cite{DLM04a}.

In this Section, we establish the corresponding theory
for the renormalized Bochner-Laplacian on symplectic orbifolds.
In \cite[\S 5.4]{MM07} one can find more explanations and references for 
Sections \ref{pbs4.1} and \ref{pbs4.2}.
For related topics about orbifolds we refer to \cite{ALR}.

This Section is organized as follows.
In Section \ref{pbs4.1} we recall the basic definitions about orbifolds.
In Section \ref{pbs4.2} we explain the asymptotic expansion
of Bergman kernel on symplectic orbifolds, which we apply in
Section \ref{pbs4.3} to derive the Berezin-Toeplitz quantization on 
symplectic orbifolds.

\subsection{Basic definitions on orbifolds} \label{pbs4.1}

We define at first a  category $\mathcal{M}_s$ as follows~:
The objects of $\mathcal{M}_s$
are the class of pairs $(G,M)$ where $M$ is a connected smooth manifold
and $G$ is a finite group acting effectively on $M$
(i.e., if $g\in G$ such that $gx=x$ for any $x\in M$,
then $g$ is the unit element of $G$). If $(G,M)$ and
 $(G',M')$ are two objects, then a morphism $\Phi: (G,M)\rightarrow
(G',M')$ is a family of open embeddings $\varphi: M\rightarrow M'$
satisfying~:

{i)} For  each $\varphi \in \Phi $, there is an injective group
 homomorphism
$\lambda_{\varphi}~: G\rightarrow G' $ that makes $\varphi$ be
$\lambda_{\varphi}$-equivariant.

{ii)} For $g\in G', \varphi \in \Phi $, we define
$g\varphi : M \rightarrow M'$
by $(g\varphi)(x) = g\varphi(x)$ for $x\in M$.
If $(g\varphi)(M) \cap \varphi(M) \neq \emptyset$,
 then $g\in  \lambda_{\varphi}(G)$.

{ iii)} For $\varphi \in \Phi $, we have $\Phi = \{g\varphi : g\in G'\}$.

\begin{defn}\label{pbt4.1} Let $X$ be a paracompact Hausdorff space.
A $m$-dimensional \emph{orbifold chart} on $X$ consists of a
connected open set $U$ of $X$,
an object $(G_U,\widetilde{U})$ of $\mathcal{M}_s$ with 
$\dim\widetilde{U}=m$,
and a ramified covering $\tau_U:\widetilde{U}\to U$ which is
$G_U$-invariant and induces a homeomorphism
$U \simeq \widetilde{U}/G_U$. We denote the chart by
$(G_U,\widetilde{U})\stackrel{\tau_U}{\longrightarrow}U$.

A $m$-dimensional \emph{orbifold atlas} $\mathcal{V}$ on $X$
consists of a family of $m$-dimensional orbifold charts
$\mathcal{V} (U)
=((G_U,\widetilde{U})\stackrel{\tau_U}{\longrightarrow} U)$
satisfying the following conditions\,:

{i)} The open sets $U\subset X$ form a covering $\mathcal{U}$ with
 the property:
\begin{equation}\label{pb4.1}
 \text{For any $U, U'\in \mathcal{U}$ and $x\in U\cap U'$,
there is
$U''\in \mathcal{U}$ such that $x\in U''\subset U\cap U'$}.
\end{equation}

{ii)} for any $U, V\in \mathcal{U}, U\subset V$ there exists a morphism
$\varphi_{VU}:(G_U,\widetilde{U})\rightarrow (G_V,\widetilde{V})$,
which covers the inclusion $U\subset V$ and satisfies
$\varphi_{WU}=\varphi_{WV} \circ \varphi_{VU}$
for any $U,V,W\in \mathcal{U}$, with $U\subset V \subset W$.
\end{defn}

It is easy to see that there exists a unique maximal orbifold atlas
$\mathcal{V}_{max}$ containing $\mathcal{V}$;
$\mathcal{V}_{max}$ consists of all orbifold charts
$(G_U,\widetilde{U})\stackrel{\tau_U}{\longrightarrow} U$,
which are locally isomorphic to charts from $\mathcal{V}$ in
the neighborhood
of each point of $U$. A maximal orbifold atlas $\mathcal{V}_{max}$
is called an \emph{orbifold structure} and the pair 
$(X,\mathcal{V}_{max})$
is called an orbifold. As usual, once we have an orbifold atlas
$\mathcal{V}$ on $X$ we denote the orbifold by $(X,\mathcal{V})$,
since $\mathcal{V}$ determines uniquely $\mathcal{V}_{max}$\,.

Note that if $\mathcal{U}^\prime$ is a refinement of $\mathcal{U}$
satisfying \eqref{pb4.1}, then there is an orbifold
atlas $\mathcal{V}^\prime$ such that
$\mathcal{V} \cup \mathcal{V}^\prime $
is an orbifold atlas, hence
$\mathcal{V} \cup \mathcal{V}^\prime \subset \mathcal{V}_{max}$.
This shows that we may choose $\mathcal{U}$ arbitrarily fine.

Let $(X,\mathcal{V})$ be an orbifold. For each $x\in X$, we can choose a
small neighborhood $(G_x, \widetilde{U}_x)\to U_x$ such
that $x\in \widetilde{U}_x$ is a fixed point of $G_x$
(it follows from the definition that such a $G_x$ is unique up to
isomorphisms for each $x\in X$).
We denote by  $|G_x|$ the cardinal of $G_x$.
If $|G_x|=1$, then $X$ has a smooth manifold structure 
in the neighborhood of $x$, which is called a smooth point of $X$.
If  $|G_x|>1$, then $X$ is not a smooth manifold in 
the neighborhood of $x$,
which is called a singular point of $X$. We denote by
$X_{sing}= \{x\in X; |G_x|>1\}$ the singular set of $X$,
and $X_{reg}= \{x\in X; |G_x|=1\}$ the regular set of $X$.

It is useful to note that on an orbifold $(X,\mathcal{V})$ we can
construct partitions of unity. First, let us call a function on $X$ smooth,
if its lift to any chart of the orbifold atlas $\mathcal{V}$ is smooth
in the usual sense. Then the definition and construction of a smooth
partition of unity associated to a locally finite covering carries over
easily from the manifold case. The point is to construct
smooth $G_U$-invariant functions with compact support on
$(G_U,\widetilde{U})$.

In Definition \ref{pbt4.1} we can replace $\mathcal{M}_s$ by a category
of manifolds with an additional structure such as orientation,
Riemannian metric, almost-complex structure or complex structure.
We impose that the morphisms (and the groups) preserve the specified
structure. So we can define oriented, Riemannian,
almost-complex or complex orbifolds.

Let $(X,\mathcal{V})$ be an arbitrary orbifold. By the above definition,
a \emph{Riemannian metric} on $X$ is a Riemannian metric $g^{TX}$
on $X_{reg}$ such that the lift of $g^{TX}$ to any chart of the orbifold
atlas $\mathcal{V}$ can be extended to a smooth Riemannian metric.
Certainly, for any $(G_U, \wi{U})\in \mathcal{V}$, we can always
construct a $G_U$-invariant Riemannian metric on $\wi{U}$.
By a partition of unity argument, we see that there exist 
Riemannian metrics on the orbifold $(X,\mathcal{V})$.

\begin{defn}\label{pbt4.4}
 An orbifold vector bundle $E$ over an
orbifold $(X,\mathcal{V})$ is defined as follows\,: $E$ is an orbifold
and for $U\in \mathcal{U}$, $(G_U^{E}, \widetilde{p}_U:
\widetilde{E}_U \rightarrow \widetilde{U})$ is
 a $G_U^{E}$-equivariant vector bundle, $(G_U^{E}, \widetilde{E}_U)$
(resp. $(G_U=G_U^{E}/K_U^{E}, \widetilde{U})$,
$K_U^{E}= \ke (G_U^{E}\rightarrow\mbox{{\rm Diffeo}} (\widetilde{U})))$
is the orbifold structure of $E$ (resp. $X$) and morphisms for 
$(G_U^{E}, \widetilde{E}_U)$ are morphisms of equivariant vector bundles.
If $G_U^E$ acts effectively on $\widetilde{U}$ for $U\in \mathcal{U}$,
i.e., $K_U^{E} = \{ 1\}$, we call $E$ a proper orbifold vector bundle.
\end{defn}

Note that any structure on $X$ or $E$ is locally
$G_x$ or $G_{U_x}^{E}$-equivariant.

\begin{rem}\label{pbt4.5}
Let $E$ be an orbifold vector bundle on $(X,\mathcal{V})$.
For $U\in \mathcal{U}$,
let $\wi{E ^{\pr}_U}$ be the maximal $K_U^{E}$-invariant sub-bundle of
 $\wi{E}_U$ on $\wi{U}$. Then $(G_U, \wi{E ^{\pr}_U})$  defines a
proper orbifold vector bundle on $(X, \mathcal{V})$, denoted by $E ^{\pr}$.

The  (proper)  orbifold tangent bundle $TX$ on an orbifold $X$ is defined by
$(G_U, T\widetilde{U} \rightarrow \widetilde{U})$, for $U\in \mathcal{U}$.
In the same vein we introduce the cotangent bundle $T^*X$.
We can form tensor products of bundles by taking the tensor
products of their local expressions in the charts of an orbifold atlas.
Note that a Riemannian metric on $X$ induces a section of 
$T^*X\otimes T^*X$ over $X$
which is a positive definite bilinear form on $T_xX$ at each point $x\in X$.
\end{rem}

Let $E \rightarrow X$ be an orbifold vector bundle and
$k\in \NN\cup\{\infty\}$. A section $s: X\rightarrow E$ is called 
$\mathscr{C}^k$
 if for each $U\in \mathcal{U}$, $s|_{U}$ is covered by
a $G_U^{E}$-invariant $\mathscr{C}^k$
section $\widetilde{s}_U : \widetilde{U} \rightarrow \widetilde{E}_U$.
 We denote by $\mathscr{C}^k(X,E)$
the space of $\mathscr{C}^k$ sections of $E$ on $X$.

If $X$ is oriented, we define the integral
 $\int_{X} \alpha$ for a form $\alpha$
 over $X$ (i.e., a section of $ \Lambda  (T^*X)$ over $X$) as follows.
If $\supp(\alpha) \subset U\in \mathcal U$, set
\begin{equation}\label{pb4.5}
\int_{X} \alpha: = \frac{1}{|G_U|} \int_{\widetilde{U}}
\widetilde{\alpha}_U.
\end{equation}
It is easy to see that the definition is independent of the chart.
For general $\alpha$ we extend the definition by using a partition of unity.

If $X$ is an oriented Riemannian orbifold,
there exists a canonical volume element $dv_X$ on $X$, which
is a section of
$\Lambda^m(T^*X)$, $m=\dim X$. Hence, we can also integrate
functions on $X$.

Assume now that the Riemannian orbifold $(X,\mathcal V)$ is compact.
For $x,y\in X$, put
\begin{eqnarray}\begin{array}{l}
d(x,y) = \mbox{Inf}_\gamma  \Big \{ \sum_i \int_{t_{i-1}}^{t_i}
|\frac{\partial }{\partial t}\widetilde{\gamma}_i(t)| dt \Big |
 \gamma: [0,1] \to X, \gamma(0) =x, \gamma(1) = y,\\
 \hspace*{15mm}  \mbox{such that there exist }   
 t_0=0< t_1 < \cdots < t_k=1,
\gamma([t_{i-1}, t_i])\subset U_i,  \\
\hspace*{15mm}U_i \in \mathcal{U}, \mbox{ and a } \mathscr{C}^{\infty}
\mbox{ map  } \widetilde{\gamma}_i: [t_{i-1}, t_i] \to \widetilde{U}_i
 \mbox{ that covers } \gamma|_{[t_{i-1}, t_i]}   \Big \}.
\end{array}\nonumber
\end{eqnarray}
Then $(X, d)$ is a metric space.

Let us discuss briefly kernels and operators on orbifolds.
For any open set $U\subset X$ and orbifold chart
$(G_U,\widetilde{U})\stackrel{\tau_U}{\longrightarrow} U$,
we will add a superscript $\, \wi{}\,$ to indicate the corresponding
objects on $\widetilde{U}$.
Assume that
$\wi{\mK}(\wi{x},\wi{x}^{\,\prime})\in \mathscr{C}^\infty(\wi{U}
\times \wi{U},
\pi_1^* \wi{E}\otimes \pi_2^* \wi{E}^*)$
verifies
\begin{align}\label{pb4.3}
(g,1)\wi{\mK}(g^{-1}\wi{x},\wi{x}^{\,\prime})
=(1,g^{-1})\wi{\mK}(\wi{x},g\wi{x}^{\,\prime})
\quad \text{ for any $g\in G_{U}$,}
\end{align}
where $(g_1,g_2)$ acts on
$\wi{E}_{\wi{x}}\times \wi{E}_{\wi{x}^{\,\prime}}^*$
by $(g_1,g_2)(\xi_1,\xi_2)= (g_1\xi_1,g_2 \xi_2)$.

We define the operator $\wi{\mK}: \mathscr{C}^\infty_0(\wi{U}, \wi{E})
\to \mathscr{C}^\infty(\wi{U}, \wi{E})$ by
\begin{equation}\label{pb4.4}
(\wi{\mK}\, \wi{s}) (\wi{x})= \int_{\wi{U}} \wi{\mK}(\wi{x},
\wi{x}^{\,\prime})
 \wi{s}(\wi{x}^{\,\prime}) dv_{\wi{U}} (\wi{x}^{\,\prime})
\quad\text{for $\wi{s} \in \mathscr{C}^\infty_0(\wi{U}, \wi{E})$\,.}
\end{equation}
Any element $g\in G_U$ acts on $\mathscr{C}^\infty(\wi{U}, \wi{E})$ by:
$(g \cdot \wi{s})(\wi{x}):=g \cdot \wi{s}(g^{-1}\wi{x})$, 
where $\wi{s} \in \mathscr{C}^\infty(\wi{U}, \wi{E})$. 
We can then identify
an element $s \in \mathscr{C}^\infty({U}, {E})$
with an element $\wi{s} \in \mathscr{C}^\infty(\wi{U}, \wi{E})$
verifying $g\cdot \wi{s}=\wi{s}$ for any $g\in G_U$.

With this identification, we define the operator
$\mK: \mathscr{C}^\infty_0(U,{E})\to \mathscr{C}^\infty({U},{E})$ by
\begin{equation}\label{pb4.6}
({\mK} s)(x)=\frac{1}{|G_U|}\int_{\wi{U}} \wi{\mK}(\wi{x},
\wi{x}^{\,\prime})
 \wi{s}(\wi{x}^{\,\prime}) dv_{\wi{U}} (\wi{x}^{\,\prime})
\quad\text{for $s \in \mathscr{C}^\infty_0({U},{E})$\,,}
\end{equation}
where $\wi{x}\in\tau^{-1}_U(x)$.
Then the smooth kernel $\mK(x,x^\prime)$ of the operator $\mK$
with respect to $dv_X$ is (cf.\  \cite[(5.18)]{DLM04a})
\begin{equation}\label{pb4.7}
\mK(x,x^\prime)= \sum_{g\in G_U} (g,1)\wi{\mK}(g^{-1}\wi{x},
\wi{x}^{\,\prime}).
\end{equation}

Let $\mK_1,\mK_2$ be two operators as above and assume that the 
kernel of one of $\wi{\mK}_1, \wi{\mK}_2$ has compact support.
By \eqref{pb4.5}, \eqref{pb4.3}, and \eqref{pb4.6},
the kernel of $\mK_1\circ \mK_2$ is given by
\begin{equation}\label{pb4.9}
(\mK_1\circ \mK_2)(x,x^\prime)= \sum_{g\in G_U}
(g,1)(\wi{\mK}_1\circ \wi{\mK}_2)(g^{-1}\wi{x},\wi{x}^{\,\prime}).
\end{equation}

\subsection{Bergman kernel on symplectic orbifolds} \label{pbs4.2}
Let $(X,\omega)$ be a compact symplectic orbifold of dimension $2n$. 
Assume that there exists a proper orbifold Hermitian line bundle $(L,h^L)$ 
on $X$ with a Hermitian connection 
$\nabla^L : \mathscr{C}^\infty(X,L)\to 
\mathscr{C}^\infty(X,T^*X\otimes L)$ 
satisfying the prequantization condition:
\begin{equation}\label{0.-1}
\frac{i}{2\pi}R^L=\omega.
\end{equation}
where $R^L=(\nabla^L)^2$ is the curvature of $\nabla^L$.
Let $(E, h^E)$ be a proper orbifold Hermitian vector bundle on $X$ 
equipped with a Hermitian connection $\nabla^E$ and $R^E$ be
the curvature of $\nabla^E$.

Let $g^{TX}$ be a Riemannian metric on $X$. Let $\Delta_p$ be 
the renormalized Bochner-Laplacian acting on 
$\mathscr{C}^\infty(X,L^p\otimes E)$ by \eqref{e:Delta_p}.  
With the same proof as in \cite[Corollary 1.2]{ma-ma02}, 
we can establish the spectral gap property.

\begin{thm}\label{t:gap}
Let $(X,\om)$ be a compact symplectic orbifold, $(L,\nabla^{L}, h^L)$ 
be a prequantum Hermitian proper orbifold line bundle
on $(X,\om)$ and $(E,\nabla^{E},h^E)$ be an arbitrary Hermitian
proper orbifold vector bundle on $X$. There exists $C_L>0$ 
such that for any $p$
 \begin{equation}\label{pb4.17}
 \sigma(\Delta_p)\subset [-C_L,C_L]\cup [2p\mu_0-C_L,+\infty),
 \end{equation}
where $\mu_0>0$ is given by \eqref{mu}.
\end{thm}

From now on, we assume  $p>C_L(2\mu_0)^{-1}$. We consider
the subspace $\mathcal H_p\subset L^2(X,L^p\otimes E)$ 
spanned by the eigensections of $\Delta_p$ corresponding to eigenvalues
in $[-C_L,C_L]$. We define the generalized Bergman kernel
\begin{align}\label{e:5.10}
	P_p(\cdot,\cdot)\in \mathscr{C}^\infty(X\times X,
\pi_1^*(L^p\otimes E)\otimes \pi_2^*((L^p\otimes E)^*))
\end{align}
as  the smooth kernel with respect to the 
Riemannian volume form $dv_X(x')$ of the orthogonal projection
(Bergman projection) $P_{\mathcal H_p}$ from $L^2(X,L^p\otimes E)$ 
onto $\mathcal H_p$.

Consider an open set $U\subset X$ and orbifold chart
$(G_U,\widetilde{U})\stackrel{\tau_U}{\longrightarrow} U$. Recall that
we add a superscript $\, \wi{}\,$ to indicate the corresponding
objects on $\widetilde{U}$. 
The Riemannian metric $g^{TX}$ can be lifted to 
a $G_U$-invariant Riemannian metric $g^{\wi{U}}$ on $\wi{U}$. 
We denote by $B^{\wi{U}}(\wi{x},\varepsilon)$ and 
$B^{T_{\wi{x}}\wi{U}}(0,\varepsilon)$ the open balls in 
$\wi{U}$ and $T_{\wi{x}}\wi{U}$ with center $\wi x$ and $0$ 
and radius $\varepsilon$, respectively. We will always assume that
$\tau_U$ extends to 
$(G_U,\widetilde{V})\stackrel{\tau_V}{\longrightarrow} V$ 
with $U\subset\subset V$ and $\wi{U}\subset\subset \wi{V}$. 
Let $\partial U=\ov{U}\setminus U$ and $\partial \wi{U}
=\ov{\wi{U}}\setminus \wi{U}$. Fix $a^X>0$ such that 
for every open set $U\subset X$ and orbifold chart 
$(G_U,\widetilde{U})\stackrel{\tau_U}{\longrightarrow} U$,
for every $\varepsilon\leq a^X$ and for every $\wi{x}\in \wi{U}$
such that $d(\wi{x},\partial \wi{U})\leq \varepsilon$, the exponential
map $T_{\wi{x}}\wi{U}\ni Z\mapsto \exp^X_{\wi{x}}(Z)\in \wi{U}$ 
is a diffeomorphism from $B^{T_{\wi{x}}\wi{U}}(0,\varepsilon)$ onto 
$B^{\wi{U}}(\wi{x},\varepsilon)$. Throughout in what follows,
$\varepsilon$ runs in the fixed interval $(0,a^X/4)$.

Let $f : \mathbb  R \to [0,1]$ be a smooth even function such that
$f(v)=1$ for $|v| \leqslant  \varepsilon/2$,
and $f(v) = 0$ for $|v| \geqslant \varepsilon$. Set
\begin{equation} \label{0c3}
F(a)= \Big(\int_{-\infty}^{+\infty}f(v) dv\Big)^{-1}
\int_{-\infty}^{+\infty} e ^{i v a}\, f(v) dv.
\end{equation}
Then $F(a)$ is an even function and lies in the Schwartz space
$\mathcal{S} (\mathbb R)$ and $F(0)=1$.
Let $\wi{F}$ be the holomorphic function on
 $\mathbb C$ such that $\wi{F}(a ^2) =F(a)$.
The restriction of $\wi{F}$
 to $\mathbb R$ lies in the Schwartz space $\mS (\mathbb R)$. 
 Then there exists
 $\{c_j\}_{j=1}^{\infty}$ such that for any $k\in \mathbb N$, the function
\begin{equation} \label{0c5}
\wi{F}_k(a)= \wi{F}(a) - \sum_{j=1}^k c_j a ^j \wi{F}(a), \quad c_{k+1}
=\frac{1}{(k+1)!}\wi{F}_k^{(k+1)}(0),
\end{equation}
verifies
\[
\wi{F}_k^{(i)}(0)= 0 \quad \mbox{\rm for any} \  0< i\leqslant k.
 \]

Using the same arguments as in \cite{ma-ma08}, one can show the 
analog of \cite[Proposition 1.2]{ma-ma08}:

\begin{prop}\label{0t3.0}
 For any $k,m\in \mathbb N$, there exists $C_{k,m}>0$ such that for
 $p\geqslant1$
\begin{align}\label{pb4.21}
\left|\wi{F}_k\big(\tfrac{1}{\sqrt{p}}\Delta_{p}\big)(x,x')
- P_{p}(x,x')\right|_{\mathscr{C}^m(X\times X)}
\leqslant C_{k,m} p^{-\frac{k}{2} +2(2m+2n+1)}.
\end{align}
Here the $\mathscr{C}^m$ norm is induced by $\nabla^L,\nabla^E$ and 
$h^L,h^E,g^{TX}$.
\end{prop}

Using \eqref{0c3}, \eqref{0c5}, and the property of 
the finite propagation speed of solutions of hyperbolic equations
\cite[Appendix D.2]{MM07} 
(which still holds on orbifolds as pointed out in \cite{Ma05})
it is clear that for $x,x'\in X$, 
$\wi{F}_k\big(\frac{1}{\sqrt{p}}\Delta_{p}\big)(x,\cdot)$
only depends on
the restriction of $\Delta_{p}$ to $B^X(x,\varepsilon p^{-\frac{1}{4}})$,
and $\wi{F}_k\big(\frac{1}{\sqrt{p}}\Delta_{p}\big)(x,x')= 0$,
if $d(x, x') \geqslant\varepsilon p^{-\frac{1}{4}}$.

Consider an open set $U\subset X$ with an orbifold chart
$(G_U,\widetilde{U})\stackrel{\tau_U}{\longrightarrow} U$.
Let $$U_1=\{ x\in U, d(x,\partial U)<2\varepsilon\}.$$ 
For any $x_0\in U_1$,
the exponential map $\exp_{\widetilde x_0}$ is a diffeomorphism from 
$B^{T_{\widetilde x_0}\widetilde{U_1}}(0,2\varepsilon)$ onto 
$B^{\widetilde{U_1}}(\widetilde x_0,2\varepsilon)$ which is 
$G_{x_0}$-equivariant. Thus we can extend everything from 
$B^{T_{\widetilde x_0}\widetilde{U_1}}(0,2\varepsilon)$ to 
$\widetilde{X_0}:=T_{\widetilde x_0}\widetilde{U_1}$ as in 
Section~\ref{expansions} (cf.\  \cite[\S 1.2]{ma-ma08}) which
is automatically $G_{x_0}$-equivariant. Let $P_{\widetilde{X_0},p}$
be the spectral projection of the renormalized Bochner-Laplacian
$\Delta_p^{\widetilde{X_0}}$ on 
$\mathscr{C}^\infty(\widetilde{X_0}, L_0^p\otimes E_0)$, 
corresponding to the interval $[-C_{L_0},C_{L_0}]$, and 
$P_{\widetilde{X_0},p}(\wi{x},\wi{y})$ be the Schwartz kernel of 
$P_{\widetilde{X_0},p}$ with respect to the volume form 
$dv_{\widetilde{X_0}}$. Let $P_{X_0,p}(x,y)$ be the corresponding 
object on $G_{x_0}\setminus T_{\widetilde x_0}\widetilde{U_1}$, then, 
by \eqref{pb4.7}, we have
\begin{equation}\label{pb4.7a}
P_{X_0,p}(x,y)= \sum_{g\in G_{x_0}} (g,1)
P_{\widetilde{X_0},p}(g^{-1}\wi{x},\wi{y}^{\,\prime}).
\end{equation}
By Proposition \ref{0t3.0}, we get the analog of
\cite[Proposition 1.3]{ma-ma08}: for any $\ell, m\in\mathbb N$,
there exists $C_{\ell,m}>0$ such that, for any $x,y \in B(x_0,\varepsilon)$ 
and $p\in \mathbb N^{*}$
\begin{equation}\label{pb4.7b}
|P_{p}(x,y)-P_{X_0,p}(x,y)|\leq C_{\ell,m}p^{-\ell}.
\end{equation}

\subsection{Berezin-Toeplitz quantization on symplectic orbifolds}
\label{pbs4.3}

We apply now the results of Section \ref{pbs4.2} to establish
the Berezin-Toeplitz quantization on symplectic orbifolds
by using as quantum spaces the spaces $\mH_p$.
We use the notations and assumptions of that Section.
We will closely follow and slightly modify the arguments 
of \cite[\S 6.3, 6.4]{ma-ma08a}.

Thus we have the following definition.
\begin{defn}\label{toet6.0}
A \textbf{\em Toeplitz operator}
is a sequence $\{T_p\}_{p\in \mathbb N}$ of bounded linear operators
\begin{equation}\label{toe6.1}
T_{p}:L^2(X, L^p\otimes E)\longrightarrow L^2(X, L^p\otimes E)\,,
\end{equation}
satisfying the following conditions:
\begin{description}
\item[(i)] For any $p\in \mathbb N$, we have
\begin{align}\label{e:5.17}
T_p=P_{\mathcal H_p}T_pP_{\mathcal H_p}.
\end{align}
\item[(ii)] There exists a sequence 
$g_l\in \mathscr{C}^\infty(X,\operatorname{End}(E))$ such that
\begin{align}\label{e:5.18}
T_p=P_{\mathcal H_p}\Big(\sum_{l=0}^\infty p^{-l}g_l\Big)
P_{\mathcal H_p}+\mathcal O(p^{-\infty}),
\end{align}
that is, for any  $k\geq 0$ there exists $C_k>0$ such that
\begin{align}\label{e:5.19}
\Big\|T_p-P_{\mathcal H_p}\Big(\sum_{l=0}^k p^{-l}g_l\Big)
P_{\mathcal H_p}\Big\|\leq C_kp^{-k-1}.
\end{align}
\end{description}
\end{defn}
For any section $f\in \mathscr{C}^{\infty}(X,\End(E))$,
  the  \textbf{\em Berezin-Toeplitz quantization} of $f$ is defined by
\begin{equation}\label{toe6.3}
T_{f,\,p}:L^2(X,L^p\otimes E)\longrightarrow L^2(X,L^p\otimes E)\,,
\quad T_{f,\,p}=P_p\,f\,P_p\,.
\end{equation}
The Schwartz kernel of $T_{f,\,p}$ is given by
\begin{equation} \label{toe2.5}
T_{f,\,p}(x,x')=\int_XP_p(x,x'')f(x'')P_p(x'',x')\,dv_X(x'')\,.
\end{equation}

\begin{lemma} \label{toet6.1}
For any $\varepsilon_0>0$ and any $l,m\in\NN$ there exists 
$C_{l,m,{\varepsilon}}>0$ such that
\begin{equation} \label{toe6.4}
\big|T_{f,\,p}(x,x')\big|_{\mathscr{C}^m(X\times X)}
\leqslant C_{l,m,{\varepsilon}}p^{-l}
\end{equation}
for all $p\geqslant 1$ and all $(x,x')\in X\times X$
with $d(x,x')>\varepsilon_0$.
\end{lemma}

\begin{proof}
By using Proposition \ref{0t3.0}, the proof is exactly the same as the proof
of \cite[Lemma 4.2]{ma-ma08a}.
\end{proof}

Next we obtain the asymptotic expansion of  the kernel 
$T_{f,\,p}(x,x')$ in a neighborhood of the diagonal. We will 
use  \cite[Condition 6.7]{ma-ma08a}.

Recall that the (proper) orbifold tangent bundle $TX$ on an orbifold $X$
is defined by a family of $G_U$-equivariant vector bundles
$(G_U, T\widetilde{U} \rightarrow \widetilde{U})$, for
$U\in \mathcal{U}$.
Consider the fiberwise product 
$TX\times_X TX=\{(Z,Z^\prime)\in 
T_{x_0}X\times T_{x_0}X : x_0\in X\}$. 
Let $\pi : TX\times_X TX\to X$ be the natural projection given by  
$\pi(Z,Z^\prime)=x_0$. We say that 
$Q_{r,\,x_0}\in \End( E)_{x_0}[\wi{Z},\wi{Z}^{\prime}]$,
if for any $U\in \mathcal U$, it induces a smooth section 
$Q_{r,\,x_0}\in \End( E)_{x_0}[\wi{Z},\wi{Z}^{\prime}]$ 
of the vector bundle $\pi^*(\operatorname{End}(E))$ on 
$T\wi{U}\times_{\wi{U}} T\wi{U}$ defined for all $\wi{x}_0\in \wi{U}$ 
and $Z,Z^\prime\in T_{\wi{x}_0}\wi{U}$ and polynomial in 
$Z,Z^\prime\in T_{\wi{x}_0}\wi{U}$.

Let $\{\Xi_p\}_{p\in\NN}$ be a sequence of linear operators
$\Xi_p: L^2(X,L^p\otimes E)\longrightarrow L^2(X,L^p\otimes E)$
with smooth kernel $\Xi_p(x,y)$ with respect to $dv_X(y)$.

\begin{condition}\label{coe2.71}
Let $k\in\NN$. Assume that there exists a family
$\{Q_{r,\,x_0}\}_{0\leqslant r\leqslant k,\,x_0\in X}$, satisfying
the conditions:
\begin{itemize}
\item[(a)] $Q_{r,\,x_0}\in \End( E)_{x_0}[\wi{Z},\wi{Z}^{\prime}]$\,, 
$\wi{Z},\wi{Z}^\prime\in T_{\wi{x}_0}X$,
\item[(b)] $\{Q_{r,\,x_0}\}_{x_0\in X}$ is smooth with respect
to the parameter $x_0\in X$,
\end{itemize}
and, for every open set $U\in\mathcal{U}$ and every orbifold chart
$(G_U,\widetilde{U})\stackrel{\tau_U}{\longrightarrow} U$,
a sequence of kernels
\[
\{\wi{\Xi}_{p, U}(\wi{x},\wi{x}^{\,\prime})\in
\mathscr{C}^\infty(\wi{U} \times \wi{U},
\pi_1^* (\wi{L}^p\otimes \wi{E})\otimes
\pi_2^* (\wi{L}^p\otimes\wi{E})^*)\}_{p\in\NN}
\]
such that, for every $\varepsilon''>0$ and every
$\wi{x},\wi{x}^{\,\prime}\in \wi{U}$,
\begin{equation} \label{toe6.5}
\begin{split}
& (g,1)\wi{\Xi}_{p, U}(g^{-1}\wi{x},\wi{x}^{\,\prime})=
(1,g^{-1})\wi{\Xi}_{p, U}(\wi{x},g\wi{x}^{\,\prime})
\quad \text{for any } \, \, g\in G_{U}\; \text{(cf.\  \eqref{pb4.3})}, \\
&\wi{\Xi}_{p, U}(\wi{x},\wi{x}^{\,\prime})= \cO(p^{-\infty}) \quad
 \quad\text{for}\, \,   d(x,x^\prime)>\varepsilon'',\\
&\Xi_{p}(x,x^\prime)
= \sum_{g\in G_{U}} (g,1)\wi{\Xi}_{p, U}(g^{-1}\wi{x},\wi{x}^{\,\prime})
+ \cO(p^{-\infty}),
\end{split}\end{equation}
and moreover, for every relatively compact open subset 
$\wi{V}\subset \wi{U}$, the following relation is valid for any 
$\wi{x}_0\in \wi{V}$:
\begin{equation} \label{toe6.6}
p^{-n}\, \wi{\Xi}_{p,U,\wi{x}_0}(\wi{Z},\wi{Z}^\prime)
\cong \sum_{r=0}^k (Q_{r,\,\wi{x}_0} \mathcal P_{\wi{x}_0})
(\sqrt{p}\wi{Z},\sqrt{p}\wi{Z}^{\prime})p^{-\frac{r}{2}}
+\mO(p^{-\frac{k+1}{2}})\,.
\end{equation}
which means that there exist $\varepsilon^\prime>0$ and $C_0>0$ 
with the following property:
for any $m\in \mathbb N$, there exist $C>0$ and $M>0$ such that for 
any $\wi{x}_0\in \wi{V}$, $p\geq 1$ and $\wi{Z},\wi{Z}^\prime
\in T_{\wi{x}_0}\wi{U}$, $|\wi{Z}|, |\wi{Z}^\prime|<\varepsilon^\prime$, 
we have with $\kappa$  in \eqref{e:kappa},
\begin{multline}\label{e:5.25}
\Bigg|p^{-n}\, \wi{\Xi}_{p,U,\wi{x}_0}(\wi{Z},\wi{Z}^\prime)
\kappa^{\frac 12}(\wi{Z})\kappa^{\frac 12}(\wi{Z}^\prime)
-\sum_{r=0}^k(Q_{r,\,\wi{x}_0} \mathcal P_{\wi{x}_0})
(\sqrt{p}\wi{Z},\sqrt{p}\wi{Z}^{\prime})p^{-\frac{r}{2}}
\Bigg|_{\mathscr{C}^{m}(\wi{V})}\\
\leq Cp^{-\frac{k+1}{2}}(1+\sqrt{p}|\wi{Z}|+\sqrt{p}|\wi{Z}^\prime|)^M
\exp(-\sqrt{C_0p}|\wi{Z}-\wi{Z}^\prime|)+\mathcal O(p^{-\infty}).
\end{multline}
\end{condition}

\begin{notation}\label{noe2.71}
If the sequence $\{\Xi_p : L^2(X,L^p\otimes E)\longrightarrow
L^2(X,L^p\otimes E)\}_{p\in\NN}$ satisfies Condition \ref{coe2.71},
 we write
\begin{equation} \label{toe6.7}
p^{-n}\, \Xi_{p,\,x_0}(Z,Z^\prime)\cong \sum_{r=0}^k
(Q_{r,\,x_0} \mathcal P_{x_0})(\sqrt{p}Z,\sqrt{p}Z^{\prime})p^{-\frac{r}{2}}
+\mO(p^{-\frac{k+1}{2}})\,.
\end{equation}
\end{notation}

As in \cite[Lemma 6.9]{ma-ma08a}, one can show that the smooth family
$Q_{r,\,x_0}\in \End( E)_{x_0}[\wi{Z},\wi{Z}^{\prime}]$
in Condition \ref{coe2.71} is uniquely determined by $\Xi_p$.

\begin{thm} \label{toet2.2}
There exist polynomials 
$J_{r,\,x_0}\in \End(E)_{x_0}[\wi{Z},\wi{Z}^{\prime}]$ 
such that, for any $k\in \NN$, $Z,Z^\prime \in T_{x_0}X$, 
$\abs{Z},\abs{Z^{\prime}}< \varepsilon$,
we have
\begin{equation} \label{toe2.9}
p^{-n} P_{p,\,x_0}(Z,Z^\prime) \cong \sum_{r=0}^{k-1}
(J_{r,\,x_0} \mathcal P_{x_0})(\sqrt{p}Z,\sqrt{p}Z^{\prime})p^{-\frac{r}{2}}
+\mO(p^{-\frac{k}{2}})\,,
\end{equation}
 in the sense of Notation \ref{noe2.71}.
\end{thm}

\begin{proof}
By Theorem~\ref{t:main} for $P_{\widetilde{X_0},p}(\wi{x},\wi{y})$, 
\eqref{pb4.7a} and \eqref{pb4.7b}, we get Theorem~\ref{toet2.2} 
as the analog of \cite[Lemmas 4.5, 6.10]{ma-ma08a}.
\end{proof}

From \eqref{toe2.5} and \eqref{toe2.9}, we deduce an analog 
of \cite[Lemmas 4.6, 4.7 and 6.10]{ma-ma08a}.

\begin{lemma} \label{toet2.3}
Let $f\in \mathscr{C}^\infty(X,\End(E))$.
There exists a family $\{Q_{r,\,x_0}(f)\}_{r\in\NN,\,x_0\in X}$ such that
\begin{itemize}
\item[(a)] $Q_{r,\,x_0}(f)\in\End(E)_{x_0}[Z,Z^{\prime}]$
are polynomials with the same parity as $r$,
\item[(b)] $\{Q_{r,\,x_0}(f)\}_{r\in\NN,\,x_0\in X}$ is smooth with
respect to $x_0\in X$,
\item[(c)] for every $k\in \NN$, $x_0\in X$,
 $Z,Z^\prime \in T_{x_0}X$, $\abs{Z},\abs{Z^{\prime}}<\varepsilon/2$ 
 we have
\begin{equation} \label{toe2.13}
p^{-n}T_{f,\,p,\,x_0}(Z,Z^{\prime})
\cong \sum^k_{r=0}(Q_{r,\,x_0}(f)\mathcal P_{x_0})
(\sqrt{p}Z,\sqrt{p}Z^{\prime})
p^{-\frac{r}{2}} + \mO(p^{-\frac{k+1}{2}})\,,
\end{equation}
in the sense of Notation \ref{noe2.71}.
\end{itemize}
Here $Q_{r,\,x_0}(f)$ are expressed by
\begin{equation} \label{toe2.14}
Q_{r,\,x_0}(f) = \sum_{r_1+r_2+|\alpha|=r}
  \cK\Big[J_{r_1,\,x_0}\;,\;
\frac{\partial ^\alpha f_{\,x_0}}{\partial Z^\alpha}(0)
\frac{Z^\alpha}{\alpha !} J_{r_2,\,x_0}\Big]\,.
\end{equation}
Especially,
\begin{align} \label{toe2.15}
Q_{0,\,x_0}(f)= f(x_0) .
\end{align}
\begin{equation} \label{toe2.20}
Q_{1,\,x_0}(f)= f(x_0) J_{1,\,x_0}  + \cK\Big[J_{0,\,x_0},
\frac{\partial f_{x_0}}{\partial Z_j}(0) Z_j J_{0,\,x_0}\Big].
\end{equation}
\end{lemma}

Here, for any polynomials $F,G\in \mathbb C[Z,Z^\prime]$, 
the polynomial $\mathcal K[F,G]\in \mathbb C[Z,Z^\prime]$
is defined by the relation
\[
((F\mathcal P_{x_0})\circ (G\mathcal P_{x_0}))(Z,Z^\prime)=(\mathcal K[F,G]
\mathcal P_{x_0})(Z,Z^\prime),
\]
where $((F\mathcal P_{x_0})\circ (G\mathcal P_{x_0}))(Z,Z^\prime)$ is the kernel 
of the composition $(F\mathcal P_{x_0})\circ (G\mathcal P_{x_0})$ of the operators 
$F\mathcal P_{x_0}$ and $G\mathcal P_{x_0}$ in $L^2(T_{x_0}X)$ with kernels 
$(F\mathcal P_{x_0})(Z,Z^\prime)$ and $(G\mathcal P_{x_0})(Z,Z^\prime)$, respectively.

Now we can proceed by a word for word repetition of the corresponding 
arguments in \cite{ma-ma08a}. So we just give statements of the main
results.

First, the following analog of \cite[Theorem 6.11]{ma-ma08a} provides 
a useful criterion for a given family to be a Toeplitz operator.
\begin{thm}\label{toet6.4}
Let $\{T_p:L^2(X,L^p\otimes E)\longrightarrow L^2(X,L^p\otimes E)\}$
be a family of bounded linear operators which satisfies
the following three conditions:
\begin{itemize}
\item[(i)] For any $p\in \NN$,  
$P_{\mathcal H_p}\,T_p\,P_{\mathcal H_p}=T_p$\,.
\item[(ii)] For any $\varepsilon_0>0$ and any $l\in\NN$,
there exists $C_{l,\varepsilon_0}>0$ such that
for all $p\geqslant 1$ and all $(x,x')\in X\times X$
with $d(x,x')>\varepsilon_0$,
\begin{equation} \label{toe6.21}
|T_{p}(x,x')|\leqslant C_{l,{\varepsilon_0}}p^{-l}.
\end{equation}
\item[(iii)] There exists a family of polynomials
$\{\mQ_{r,\,x_0}\in\End(E)_{x_0}[Z,Z^{\prime}]\}_{x_0\in X}$
such that\,{\rm:}
\begin{itemize}
\item[(a)] each $\mQ_{r,\,x_0}$ has the same parity as $r$,

\item[(b)] the family is smooth in $x_0\in X$ and

\item[(c)] for every $k\in\NN$, we have
\begin{equation} \label{toe6.22}
p^{-n}T_{p,\,x_0}(Z,Z^{\prime})\cong
\sum^k_{r=0}(\mQ_{r,\,x_0}\mathcal P_{x_0})
(\sqrt{p}Z,\sqrt{p}Z^{\prime})p^{-\frac{r}{2}} + \mO(p^{-\frac{k+1}{2}})
\end{equation}
in the sense of \eqref{toe6.7}.
\end{itemize}
\end{itemize}

Then $\{T_p\}$ is a Toeplitz operator.
\end{thm}

Finally, we show that the set of Toeplitz operators on a compact orbifold is
closed under the composition of operators, so forms an algebra
(an analog of \cite[Theorems 6.13 and 6.16]{ma-ma08a}).
\begin{thm}\label{toet6.7}
Let $(X,\om)$ be a compact symplectic orbifold and $(L,\nabla^{L},h^L)$
be a Hermitian proper orbifold line bundle satisfying
 the prequantization condition \eqref{0.-1}.
Let $(E,\nabla^E ,h^E)$ be an arbitrary Hermitian proper orbifold
vector bundle on $X$.

Given $f,g\in \mathscr{C}^\infty(X,\End(E))$, the product of 
the Toeplitz operators
$T_{f,\,p}$ and  $T_{g,\,p}$ is a Toeplitz operator,
more precisely, it admits an asymptotic expansion
\begin{align}\label{e:5.33}
T_{f,\,p}T_{g,\,p}=\sum_{r=0}^\infty p^{-r}T_{C_r(f,g),\,p}
+ \mathcal O(p^{-\infty}),
\end{align}
where $C_r(f,g)\in \mathscr{C}^\infty(X,\End(E))$
and $C_r$ are bidifferential operators defined locally
on each covering $\wi{U}$ of an orbifold chart
 $(G_U,\widetilde{U})\stackrel{\tau_U}{\longrightarrow}U$.
In particular $C_0(f,g)=fg$.

If $f,g\in \mathscr{C}^\infty(X)$, then
\begin{align}\label{e:5.34}
[T_{f,p}, T_{g,p}]=\frac{\sqrt{-1}}{p}T_{\{f,g\},p}+\mathcal O(p^{-2}).
\end{align}
For any $f\in \mathscr{C}^\infty(X,\End(E))$, the norm of $T_{f,p}$ 
satisfies
\begin{align}\label{e:5.35}
\lim_{p\to\infty}\|T_{f,p}\|=\|f\|_{\infty}
:=\sup_{\substack{x\in X\\0\neq u\in E_x}}|f(x)(u)|_{h^E}/|u|_{h^E}.
\end{align}
\end{thm}

\begin{rem}\label{toet4.31} As mentioned in
	\cite[Remark 6.14]{ma-ma08a}, by
	\cite[Theorems 6.13, 6.16]{ma-ma08a},
on every compact symplectic orbifold $X$ admitting a prequantum
line bundle $(L, \nabla^{L}, h^L)$, one can define in 
a canonical way an associative star-product  $f*g
=\sum_{l=0}^\infty \hbar^{l}C_l(f,g)\in \mathscr{C}^\infty(X)[[\hbar]]$
for every $f,g\in \mathscr{C}^\infty(X)$,
called the \textbf{\emph{Berezin-Toeplitz star-product}} by using 
the kernel of the spin$^c$ Dirac operator. Moreover, $C_l(f,g)$
are bidifferential operators defined locally as in the smooth case.
Theorem \ref{toet6.7} shows that one can also use the eigenspaces 
of small eigenvalues of the renormalized Bochner-Laplacian.
\end{rem}

\end{document}